\documentclass{amsart}

\usepackage{amssymb}
\usepackage{hyperref}
\usepackage{amsmath,amscd}
\usepackage[enableskew]{youngtab}
\usepackage[all,cmtip]{xy}
\usepackage{mathrsfs}
\usepackage{extpfeil}

\newtheorem{theorem}{Theorem}[section]
\newtheorem{lemma}[theorem]{Lemma}
\newtheorem{proposition}[theorem]{Proposition}
\newtheorem{corollary}[theorem]{Corollary}

\theoremstyle{definition}

\theoremstyle{remark}
\newtheorem{remark}[theorem]{Remark}

\newcommand\qbin[3]{\left[\begin{matrix} #1 \\ #2 \end{matrix} \right]_{#3}}

\def\FF{\mathbb{F}}
\def\CC{\mathbb{C}}

\def\ZZ{\mathbb{Z}}
\def\SS{{\mathfrak S}}
\def\inv{\mathrm{inv}}
\def\maj{\mathrm{maj}}
\def\bx{\mathbf{x}}
\def\bs{\mathbf{s}}
\def\bh{\mathbf{h}}
\def\cleq{{\preccurlyeq}}
\def\Sym{\mathrm{Sym}}
\def\QSym{\mathrm{QSym}}
\def\NSym{\mathbf{NSym}}
\def\P{\mathbf{P}} 
\def\C{\mathbf{C}} 
\def\SYT{\mathrm{SYT}}
\def\SSYT{\mathrm{SSYT}}

\begin{document}

\title{0-Hecke algebra actions on coinvariants and flags}
\author{Jia Huang}
\address{School of Mathematics, University of Minnesota, Minneapolis, MN 55455, USA}
\email{huang338@umn.edu}
\thanks{The author is grateful to Victor Reiner for asking questions which led to the research, for giving many helpful suggestions, and for his support from NSF grant DMS-1001933. He is indebted to Andrew Berget, Tom Denton, Anne Schilling, and Monica Vazirani for interesting conversations and hospitality during the visit to UC Davis. He thanks Marcelo Aguiar, Sami Assaf, Nantel Bergeron, Adriano Garsia, Patricia Hersh, Liping Li, Sarah Mason, Dennis Stanton and Peter Webb for their helpful suggestions and comments.}


\begin{abstract}
The $0$-Hecke algebra $H_n(0)$ is a deformation of the group algebra of the symmetric group $\SS_n$. We show that its coinvariant algebra naturally carries the regular representation of $H_n(0)$, giving an analogue of the well-known result for $\SS_n$ by Chevalley-Shephard-Todd. By investigating the action of $H_n(0)$ on coinvariants and 
flag varieties, we interpret the generating functions counting the permutations with fixed inverse descent set by their inversion number and major index. We also study the action of $H_n(0)$ on the cohomology rings of the Springer fibers, and similarly interpret the (noncommutative) Hall-Littlewood symmetric functions indexed by hook shapes.
\end{abstract}

\maketitle

\section{Introduction}

A composition $\alpha$ of $n$ gives rise to a descent class of permutations in $\mathfrak S_n$; the cardinality of this descent class is known as the \emph{ribbon number} $r_\alpha $ and its inv-generating function is the \emph{$q$-ribbon number} $r_\alpha (q)$. Reiner and Stanton \cite{ReinerStanton} defined a \emph{$(q,t)$-ribbon number} $r_\alpha (q,t)$, and gave an interpretation by representations of  the symmetric group $\mathfrak S_n$ and the finite general linear group $GL(n,\mathbb F_q)$. The main goal of this work is to obtain similar interpretations of various ribbon numbers by representations of the \emph{$0$-Hecke algebra $H_n(0)$ of type $A$}, a deformation of the group algebra of $\SS_n$ over an arbitrary field $\FF$. It is defined as an associative $\FF$-algebra generated by $T_1,\ldots,T_{n-1}$ with relations
\[ \begin{cases}
T_i^2=-T_i,\ 1\leq i\leq n-1,\\
T_iT_{i+1}T_i=T_{i+1}T_iT_{i+1},\ 1\leq i\leq n-2,\\
T_iT_j=T_jT_i,\ |i-j|>1.
\end{cases} \]
It has an $\FF$-basis $\{T_w: w\in\SS_n\}$ where $T_w:=T_{i_1}\cdots T_{i_k}$ for any reduced expression $w=s_{i_1}\cdots s_{i_k}$. 

When $\FF$ is an algebraically closed field of characteristic zero, it is well-known that the simple $\SS_n$-modules $S_\lambda$ are indexed by partitions $\lambda\vdash n$, and every $\SS_n$-module is a direct sum of simple $\SS_n$-modules; the \emph{Frobenius characteristic map $\mathrm{ch}$} sends a direct sum of simple $\SS_n$-modules to the sum of the corresponding Schur functions $s_\lambda$. There is an analogue for $H_n(0)$-modules which holds over an arbitrary field $\FF$. Norton~\cite{Norton} showed that
\begin{equation}\label{eq:Norton}
H_n(0)=\bigoplus_{\alpha\models n} \P_\alpha
\end{equation}
summed over all compositions $\alpha$ of $n$, where every $\P_\alpha$ is a (left) indecomposable $H_n(0)$-module. It follows that $\{\P_\alpha: \alpha\models n\}$ is a complete list of non-isomorphic projective indecomposable $H_n(0)$-modules, and $\{\C_\alpha: \alpha\models n\}$ is a complete list of non-isomorphic simple $H_n(0)$-modules, where $\C_\alpha={\rm top}(\P_\alpha)=\P_\alpha/\,{\rm rad}\,\P_\alpha$ for all compositions $\alpha\models n$.

By Krob and Thibon~\cite{KrobThibon}, the finitely generated $H_n(0)$-modules correspond to the quasisymmetric functions via the \emph{quasisymmetric characteristic} Ch which sends a simple $\C_\alpha$ to the \emph{fundamental quasisymmetric function} $F_\alpha$, and the finitely generated projective $H_n(0)$-modules correspond to the noncommutative symmetric functions via the \emph{noncommutative characteristic} $\mathbf{ch}$ which sends a projective indecomposable $\P_\alpha$ to the \emph{noncommutative ribbon Schur function} $\bs_\alpha$. There are also graded versions of the two  characteristic maps Ch and $\mathbf{ch}$ for finitely generated (projective) $H_n(0)$-modules with filtrations.

The symmetric group $\SS_n$ acts on the polynomial ring $\FF[\bx]:=\FF[x_1,\ldots,x_n]$ by permuting the variables, and the $0$-Hecke algebra $H_n(0)$ acts on $\FF[\bx]$ via the \emph{Demazure operators} $\pi_1,\ldots,\pi_{n-1}$, where 
\begin{equation}\label{eq:Demazure}
\pi_if:= \frac{x_if-x_{i+1}s_i(f)}{x_i-x_{i+1}},\quad \forall f\in\FF[\bx].
\end{equation}
The coinvariant algebra of $H_n(0)$ coincides with that of $\SS_n$, since $\pi_if=f$ if and only if $s_if=f$ for all $f\in\FF[\bx]$. The following result for $\SS_n$ is well-known.

\begin{theorem}[Chevalley-Shephard-Todd]\label{thm:coinvar}
The coinvariant algebra $\FF[\bx]/(\FF[\bx]^{\SS_n}_+)$
is isomorphic to the regular representation of $\SS_n$, 
i.e. $\SS_n$ itself as an $\SS_n$-module, if $\FF$ is a field with ${\rm char}(\FF)\nmid n$.
\end{theorem}

We give an analogue for the $0$-Hecke algebras.

\begin{theorem}\label{thm:coinvariant}
The coinvariant algebra $\FF[\bx]/(\FF[\bx]^{\SS_n}_+)$ with $H_n(0)$-action defined by (\ref{eq:Demazure}) is isomorphic to the regular representation of $H_n(0)$, where $\FF$ is an arbitrary field.
\end{theorem}

We prove this theorem by showing a similar decomposition to Norton's decomposition (\ref{eq:Norton}). This leads to an $\FF$-basis of the coinvariant algebra, which is closely related to the well-known basis of \emph{descent monomials}. My new basis consists of certain \emph{Demazure atoms} obtained by consecutively applying the operators $\overline\pi_i=\pi_i-1$ to some descent monomials. Theorem~\ref{thm:coinvariant} and its proof are also valid when $\FF$ is replaced with $\ZZ$.

It follows from Theorem~\ref{thm:coinvariant} that the coinvariant algebra has not only the grading by the degrees of polynomials, but also the filtration by the length of permutations in $\SS_n$. This completes the following picture.
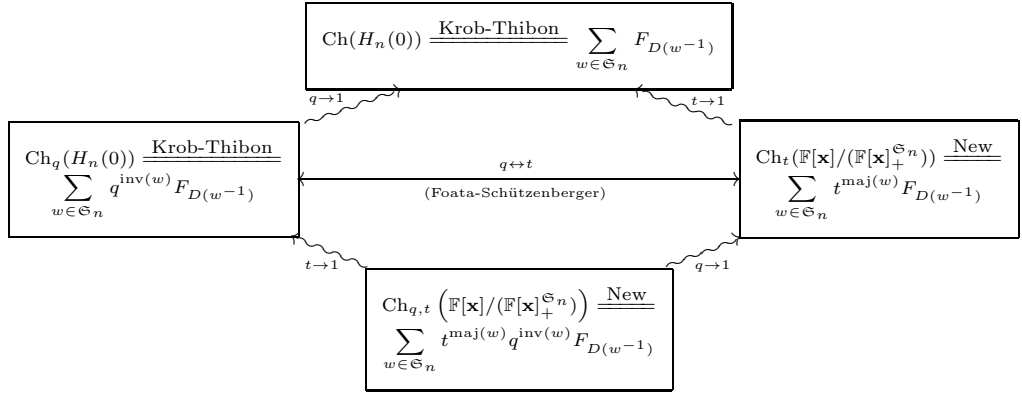
\[ \scriptsize \xymatrix @R=12pt @C=3pt{
& *++[F-]\txt{$\mathrm{Ch}(H_n(0))\xlongequal{\textrm{\scriptsize Krob-Thibon}}
{\displaystyle\sum_{w\in\mathfrak S_n}F_{D(w^{-1})}}$} \\
*++[F-]\txt{$\mathrm{Ch}_q(H_n(0))\xlongequal{\textrm{\scriptsize Krob-Thibon}}$ \\
${\displaystyle\sum_{w\in\mathfrak S_n}q^{{\rm inv}(w)}F_{D(w^{-1})}}$}
{\ar@{~>}[ru]^{q\rightarrow1} }
\ar@{<->}[rr]^{q\leftrightarrow t}_{\textrm{(Foata-Sch\"utzenberger)}}
& &
*++[F-]\txt{$\mathrm{Ch}_t(\FF[{\bx}]/(\FF[{\bx}]^{\mathfrak S_n}_+))
\xlongequal{\textrm{\scriptsize New}}$\\
${\displaystyle\sum_{w\in\mathfrak S_n} t^{{\rm maj}(w)}F_{D(w^{-1})}}$}
\ar@{~>}[lu]_{t\rightarrow1} \\
& *++[F-]\txt{$\mathrm{Ch}_{q,t}\left(\FF[{\bx}]/(\FF[{\bx}]^{\mathfrak S_n}_+)\right)
\xlongequal{\textrm{\scriptsize New}}$\\
${\displaystyle\sum_{w\in\mathfrak S_n} t^{{\rm maj}(w)}q^{{\rm
inv}(w)}F_{D(w^{-1})}}$} \ar@{~>}[lu]^--{t\rightarrow1}
\ar@{~>}[ru]_--{q\rightarrow1} \\
} \]
Here the inverse descent set $D(w^{-1})$ is identified with the composition $\alpha$ of $n$ with $D(\alpha)=D(w^{-1})$.
We shall see in Section \ref{SRepH} that $r_\alpha$ and $r_\alpha (q)$ appear as coefficients of $F_\alpha $ in $\mathrm{Ch}(H_n(0))$ and $\mathrm{Ch}_q(H_n(0))$, respectively.

Next we consider the finite general linear group $G=GL(n,\mathbb F_q)$, where $q$ is a power of a prime $p$, and its Borel subgroup $B$. The $0$-Hecke algebra $H_n(0)$ acts on the flag variety $1_B^G=\mathbb F[G/B]$ by $T_wB=BwB$ if $\mathrm{char}(\FF)=p$, and this induces an action on the coinvariant algebra $\FF[{\bx}]^B/(\FF[{\bx}]^G_+)$ of the pair $(G,B)$ (see \S\ref{sec:CoinvGB}). By studying the (graded) multiplicities of the simple $H_n(0)$-modules in these $H_n(0)$-modules, we complete the following diagram, which interprets all the ribbon numbers mentioned earlier.
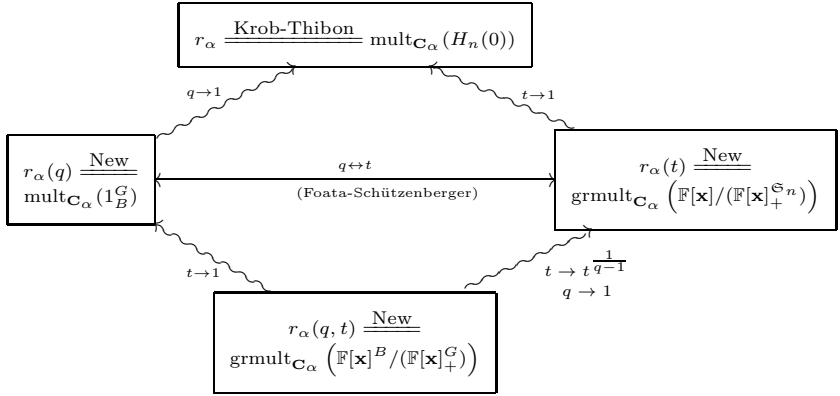
\[ \scriptsize \
\quad\xymatrix @C9pt{
& *++[F-]\txt{$r_\alpha \xlongequal{\textrm{\scriptsize Krob-Thibon}}\textrm{mult}_{\C_\alpha }(H_n(0))$} \\
*++[F-]\txt{$r_\alpha (q)\xlongequal{\textrm{\scriptsize New}}$ \\ ${\rm mult}_{\C_\alpha }(1_B^G)$}
{\ar@{~>}[ru]^{q\rightarrow1} } \ar@{<->}[rr]^--{q\leftrightarrow t}_{\textrm{(Foata-Sch\"utzenberger)}} & &
*++[F-]\txt{$r_\alpha (t)\xlongequal{\textrm{\scriptsize New}}$ \\
${\rm grmult}_{\C_\alpha }\left(\FF[{\bx}]/(\FF[{\bx}]^{\mathfrak S_n}_+)\right)$}
\ar@{~>}[lu]_{\,t\rightarrow1} \\
& *++[F-]\txt{$r_\alpha (q,t)\xlongequal{\textrm{\scriptsize New}}$\\
${\rm grmult}_{\C_\alpha }\left(\FF[{\bx}]^B/(\FF[{\bx}]^G_+)\right) $}
\ar@{~>}[lu]^{t\rightarrow1}
\ar@{~>}[ru]_>>>>>*\txt{\tiny$t\rightarrow t^{\frac1{q-1}}$\\\tiny$q\rightarrow1$} \\
}\]

Finally we consider a family of quotient rings $R_\mu=\FF[\bx]/I_\mu$ indexed by partitions $\mu$ of $n$, which contains the coinvariant algebra of $\SS_n$ as a special case ($\mu=1^n$). If $\FF=\CC$ then $R_\mu$ is isomorphic to the cohomology rings of the Springer fiber $\mathcal F_\mu$, carries an $\SS_n$-action, and has graded Frobenius characteristic equal to the modified Hall-Littlewood symmetric function $\widetilde H_\mu(x;t)$. We prove that the $H_n(0)$-action on $\FF[\bx]$ preserves the ideal $I_\mu$ if and only if $\mu$ is a hook, and if so then $R_\mu$ becomes a projective $H_n(0)$-module whose graded characteristic also equals $\widetilde{H}_\mu(x;t)$, and whose graded noncommutative characteristic equals $\widetilde{\mathbf H}_\mu(\bx;t)$, a noncommutative analogue of $\widetilde{H}_\mu(x;t)$ introduced by Bergeron and Zabrocki~\cite{BZ}.

This paper is structured as follows. First we review the definitions for the various ribbon numbers and their interpretation by representations of $\mathfrak S_n$ and $GL(n,\mathbb F_q)$ in Section~\ref{SRibbon}. Then we recall the representation theory of the $0$-Hecke algebra in Section~\ref{SRepH}. The result on the coinvariant algebra $\FF[{\bx}]/(\FF[{\bx}]^{\mathfrak S_n}_+)$ is given in Section~\ref{SCoinvariantsA}, and a generalization to the Laurent ring $\FF[\Lambda]$ of the weight lattice $\Lambda$ of a Weyl group is provided in Section~\ref{SCoinvariantsW}. Next I investigate the $0$-Hecke algebra actions on the flag variety $1_B^G$ and the coinvariant algebra of $(G,B)$ in Section~\ref{SBG} and Section~\ref{sec:CoinvGB}. The $H_n(0)$-action on the cohomology rings of the Springer fibers is studied in Section~\ref{sec:Springer}. Lastly Section~\ref{SQ} is devoted to questions for future research.

\section{Ribbon numbers}\label{SRibbon}

A \emph{composition} is a sequence $\alpha=(\alpha_1,\ldots,\alpha_\ell)$ of positive integers  $\alpha_1,\ldots,\alpha_\ell$. The \emph{length} of $\alpha$ is $\ell(\alpha):=\ell$ and the \emph{size} of $\alpha$ is $|\alpha|:=\alpha_1+\cdots+\alpha_\ell$. If the size of $\alpha$ is $n$ then say $\alpha$ is a composition of $n$ and write $\alpha\models n$. Let $\sigma_j=\alpha_1+\cdots+\alpha_j$ for $j=0,1,\ldots,\ell$; in particular, $\sigma_0=0$ and $\sigma_\ell=n$. The \emph{descent set} of $\alpha$ is $D(\alpha):=\{\sigma_1,\ldots,\sigma_{\ell-1}\}$. The map $\alpha\mapsto D(\alpha)$ is a bijection between compositions of $n$ and subsets of $[n-1]$. Write $\overleftarrow\alpha:=(\alpha_\ell,\ldots,\alpha_1)$ and let $\alpha^c$ be the composition of $n$ with $D(\alpha^c)=[n-1]\setminus D(\alpha)$. Write $\alpha\cleq\beta$ if $\alpha$ and $\beta$ are two compositions of the same size with $D(\alpha)\subseteq D(\beta)$.

A \emph{semistandard Young tableau} $\tau$ of an arbitrary skew shape $\lambda/\mu$ is a filling of the skew diagram of $\lambda/\mu$ by positive integers such that every row weakly increases from left to right and every column strictly increases from top to bottom. Reading these integers from the bottom row to the top row and proceeding from left to right within each row gives the \emph{reading word} $w(\tau)$ of $\tau$. Say $\tau$ is a \emph{standard Young tableau} if the integers appearing in $\tau$ are precisely $1,\ldots,n$ without repetition, i.e. $w(\tau)\in\SS_n$. The \emph{descents} of a standard Young tableau $\tau$ are those numbers $i$ appearing in a row strictly below $i+1$ in $\tau$, or in other words, the descents of $w(\tau)^{-1}$. The \emph{major index} $\maj(\tau)$ of a standard Young tableau $\tau$ is the sum of all its descents. Denote by $\SSYT(\lambda/\mu)$ [$\SYT(\lambda/\mu)$ resp.] the set of all semistandard [standard resp.] Young tableaux of shape $\lambda/\mu$. 

A \emph{ribbon} is a skew connected diagram without $2\times2$ boxes. A ribbon $\alpha$ whose rows have lengths $\alpha_1,\ldots,\alpha_\ell$, ordered \emph{from bottom to top}, can be identified with a composition $\alpha=(\alpha_1,\ldots,\alpha_\ell)$. Denote by $\alpha^\sim$ the \emph{transpose} of the ribbon $\alpha$. One can check that $\alpha^\sim=(\overleftarrow\alpha)^c=\overleftarrow{\alpha^c}$. An example is given below.
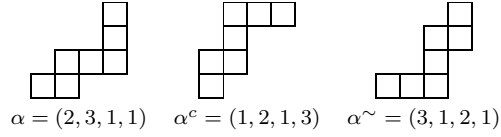
\[ \footnotesize
\begin{array}{ccc}
\young(:::\hfill,:::\hfill,:\hfill\hfill\hfill,\hfill\hfill) &
\young(:\hfill\hfill\hfill,:\hfill,\hfill\hfill,\hfill) &
\young(:::\hfill,::\hfill\hfill,::\hfill,\hfill\hfill\hfill) \\
\alpha=(2,3,1,1) & \alpha^c=(1,2,1,3) & \alpha^\sim=(3,1,2,1)
\end{array} \]

A \emph{(standard) ribbon tableau} is a standard Young tableau of ribbon shape $\alpha$. Taking the reading word $\tau\mapsto w(\tau)$ gives a bijection between $\SYT(\alpha)$ and the \emph{descent class} of $\alpha$, which consists of all permutations $w$ in $\SS_n$ with $D(w)=D(\alpha)$. The descent class of $\alpha$ is an interval under the left weak order of $\SS_n$, denoted by $[w_0(\alpha),w_1(\alpha)]$. For instance, the descent class of $\alpha=(1,2,1)$ is given below.
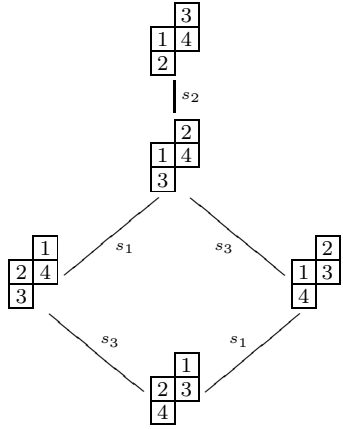
\[ \footnotesize \xymatrix @R=12pt @C=30pt {
& {\young(:3,14,2)}  \ar@{-}[d]^-{s_2}\\
& {\young(:2,14,3)}  \ar@{-}[ld]^-{s_1} \ar@{-}[rd]_-{s_3} \\
{\young(:1,24,3)} \ar@{-}[rd]^-{s_3} & &
{\young(:2,13,4)} \ar@{-}[ld]_-{s_1}\\
& {\young(:1,23,4)} 
}\]
In particular, the ribbon tableaux of $w_0(\alpha)$ and $w_1(\alpha)$ can be respectively obtained by 
\begin{itemize}
\item filling with $1,2,\ldots,n$ the columns of the ribbon $\alpha$ from top to bottom, starting with the leftmost column and proceeding toward the rightmost column,

\item filling with $1,2,\ldots,n$ the rows of the ribbon $\alpha$ from left to right, starting with the top row and proceeding toward the bottom row.
\end{itemize}

The \emph{ribbon number} $r_\alpha$ is the cardinality of the descent class of $\alpha$. The \emph{$q$-ribbon number} is 
\begin{equation}\label{eq:q-ribbon}
r_\alpha (q):=\sum_{\substack{ w\in\mathfrak S_n: \\ D(w)=D(\alpha) } } q^{\,{\rm inv}(w)}
=[n]!_q\det\left(\frac1{[\sigma_j-\sigma_{i-1}]!_q}\right)_{i,j=1}^\ell
\end{equation}
where $[n]!_q=[n]_q[n-1]_q\cdots [1]_q$ and $[n]_q=1+q+\cdots+q^{n-1}$. By Foata and Sch\"utzenberger \cite{FS},  inv and maj are equidistributed on every inverse descent class $\{w\in\SS_n:D(w^{-1})=D(\alpha)\}$. Thus 
\[ 
r_\alpha (t)=\sum_{\substack{ w\in\mathfrak S_n: \\ D(w)=D(\alpha) } } t^{\,{\rm maj}(w^{-1})} \xlongequal{ w(\tau)\leftrightarrow \tau}  \sum_{\tau\in\SYT(\alpha)} t^{\maj(\tau)}.
\]
A further generalization, introduced by Reiner and Stanton~\cite{ReinerStanton}, is the \emph{$(q,t)$-ribbon number}
\[
r_\alpha (q,t) := \sum_{\substack{ w\in\mathfrak S_n: \\ D(w)=D(\alpha) } } {\rm wt}(w;q,t) = n!_{q,t}\det\left( \varphi^{\sigma_{i-1}}\frac1{(\sigma_j-\sigma_{i-1})!_{q,t}} \right)_{i,j=1}^\ell.
\]
Here ${\rm wt}(w;q,t)$ is some weight defined by product expression, $m!_{q,t}=(1-t^{q^m-1})(1-t^{q^m-q})\cdots(1-t^{q^m-q^{m-1}})$, and $\varphi: t\mapsto t^q$ is the \emph{Frobenius operator}.


All these ribbon numbers can be interpreted by the \emph{homology representation} $\chi^\alpha $ [$\chi^\alpha_q$ resp.] of $\mathfrak S_n$ [$G=GL(n,\mathbb F_q)$ resp.], defined as the top homology of the \emph{rank-selected Coxeter complex} $\Delta(\mathfrak S_n)_\alpha$ [\emph{Tits building} $\Delta(G)_\alpha $ resp.], and by the intertwiner $M^\alpha ={\rm Hom}_{\FF\mathfrak S_n}\left(\chi^\alpha ,\FF[{\bx}]\right)$ [$M^\alpha_q={\rm Hom}_{\FF G}\left(\chi^\alpha_q,\FF[{\bx}]\right)$ resp.] as a module over $\FF[{\bx}]^{\mathfrak S_n}$ [$\FF[{\bx}]^G$ resp.]. By work of Reiner and Stanton~\cite{ReinerStanton}, we have the following picture. 
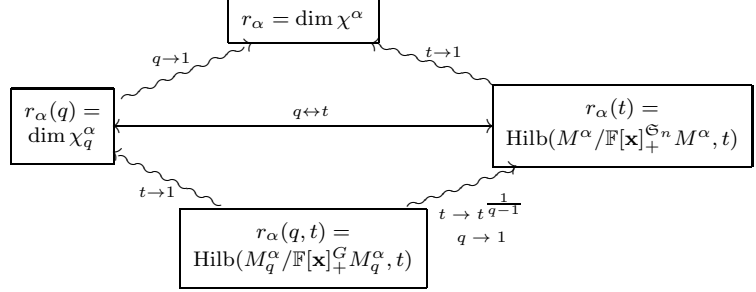
\[ \footnotesize \xymatrix @R=16pt@C=25pt {
& *++[F-]\txt{$r_\alpha =\dim\chi^\alpha $} \\
*++[F-]\txt{$r_\alpha (q)=$\\$\dim\chi^\alpha_q$}
{\ar@{~>}[ru]^-{q\rightarrow1} }
\ar@{<->}[rr]^-{\ \ q\leftrightarrow t}
& &
*++[F-]\txt{$r_\alpha (t)=$\\${\rm Hilb}(M^\alpha /\FF[{\bx}]^{\mathfrak S_n}_+M^\alpha ,t)$}
\ar@{~>}[lu]_-{t\rightarrow1} \\
& *++[F-][c]\txt{$r_\alpha (q,t)=$\\${\rm Hilb}(M^\alpha_q/\FF[{\bx}]^G_+M^\alpha_q,t)$}
\ar@{~>}[lu]^-{t\rightarrow1}
\ar@{~>}[ru]_*\txt{\tiny$t\rightarrow{t^{\frac1{q-1}}}$\\ \tiny$q\rightarrow1$\ }
}\]

\vskip6pt
The ribbon numbers are related to the multinomial coefficients by inclusion-exclusion. Let $\alpha=(\alpha_1,\ldots,\alpha_\ell)$ be a composition of $n$. Then we have the \emph{multinomial and $q$-multinomial coefficients}
\[
\qbin{n}{\alpha}{} := \frac{n!}{\alpha_1!\cdots\alpha_\ell!}=\#\{w\in\SS_n:D(w)\subseteq D(\alpha)\},
\]
\[
\qbin{n}{\alpha}{q} := \frac{[n]!_q}{[\alpha_1]!_q\cdots[\alpha_\ell]!_q} = \sum_{\substack{ w\in\SS_n: \\ D(w)\subseteq D(\alpha) }} q^{\inv(w)}.
\]
Reiner and Stanton~\cite{ReinerStanton} introduced the \emph{$(q,t)$-multinomial coefficient}
\[
\qbin{n}{\alpha}{q,t} := \frac{n!_{q,t}}{\alpha_1!_{q,t}\cdot\varphi^{\sigma_1}(\alpha_2!_{q,t})\cdots\varphi^{\sigma_{\ell-1}}(\alpha_\ell!_{q,t})} = \sum_{w\in\SS_n: D(w)\subseteq D(\alpha)} {\rm wt}(w;q,t).
\] 

Assume $q$ is a primer power below. Let $G=GL(n,\FF_q)$ be the finite general linear group over $\FF_q$, and let $P_\alpha$ be the parabolic subgroup of all invertible block upper triangular matrices whose diagonal blocks have sizes given by the composition $\alpha$. Then 
\[
\qbin{n}{\alpha}{q} = | G/P_\alpha|,
\]
\[
\qbin{n}{\alpha}{q,t}={\rm Hilb}\left(\FF[\bx]^{P_\alpha }/(\FF[\bx]^G_+),t\right).
\] 

\section{Representation theory of the $0$-Hecke algebras}\label{SRepH}

In this section we recall from Norton \cite{Norton} and Krob and Thibon~\cite{KrobThibon} the representation theory of the $0$-Hecke algebras of finite Coxeter groups. 

\subsection{Simple modules and projective indecomposable modules}
Let
\[
W=\langle\, s_1,\ldots,s_r : s_i^2=1, (s_is_js_i\cdots)_{m_{ij}}=
(s_js_is_j\cdots)_{m_{ij}},1\leq i\ne j\leq r \,\rangle
\]
be a finite Coxeter group, where $(aba\cdots)_m$ denotes an alternating product of $m$ terms. The $0$-Hecke algebra
$H_W(0)$ of $W$ is an associative $\FF$-algebra
generated by $T_1,\ldots,T_r$ with relations
\begin{equation}\label{eq:Ti}
\left\{\begin{array}{ll}
T_i^2=-T_i, & 1\leq i\leq r,\\
(T_iT_jT_i\cdots)_{m_{ij}}=(T_jT_iT_j\cdots)_{m_{ij}}, & 1\leq i\ne
j\leq r.
\end{array}\right.
\end{equation}
Another set of generators $\{T_i'=T_i+1:i=1,\ldots,r\}$ for $H_W(0)$ satisfies the relations
\begin{equation}\label{eq:T'i}
\left\{\begin{array}{ll}
(T'_i)^2=T'_i, & 1\leq i\leq r,\\
(T'_iT'_jT'_i\cdots)_{m_{ij}}=(T'_jT'_iT'_j\cdots)_{m_{ij}}, & 1\leq i\ne
j\leq r.
\end{array}\right.
\end{equation}
Thus $T_w=T_{i_1}\cdots T_{i_k}$ and $T'_w=T'_{i_1}\cdots T'_{i_k}$ are both well-defined if $w=s_{i_1}\cdots s_{i_k}$ is a reduced expression.

Let $\alpha$ be a composition of $r+1$. Similarly to the type $A$ case  (i.e. $W=\SS_n)$ mentioned in the previous section, the \emph{descent class} of $\alpha$ in $W$ consists of all elements $w\in W$ with descent set  
\[
D(w):=\{s_i\in S: \ell(ws_i)<\ell(w) \} = \{s_i : i\in D(\alpha) \}.
\]
It is an interval in the left weak order of $W$, denoted by $[w_0(\alpha),w_1(\alpha)]$, where $w_0(\alpha)$ is the longest element in the parabolic subgroup $W_{D(\alpha)}$;
see Bj\"orner and Wachs~\cite[Theorem~6.2]{BjornerWachs}. 

Norton \cite{Norton} decomposed $H_W(0)$ into a direct sum of $2^r$ non-isomorphic indecomposable submodules 
\[
\P_\alpha:=H_W(0)\cdot T_{w_0(\alpha)}T'_{w_0(\alpha^c)}
\]
indexed by compositions $\alpha$ of $r+1$; each $\P_\alpha$ has an $\FF$-basis 
\[
\left\{ T_wT'_{w_0(\alpha^c)}: w\in[w_0(\alpha),w_1(\alpha)] \right\}.
\]
The \emph{top} of $\P_\alpha$, denoted by $\C_\alpha:={\rm top}(\P_\alpha) = \P_\alpha /\,{\rm rad}\,\P_\alpha$, is a (one-dimensional) simple $H_W(0)$-module with the action of $H_W(0)$ given by
\[
T_i=\left\{\begin{array}{ll}
-1,& {\rm if}\ i\in D(\alpha),\\
0,& {\rm if}\ i\notin D(\alpha).
\end{array}\right.
\]
It follows from the general theory (see e.g. \cite[\S I.5]{ASS}) that $\{\P_\alpha:\alpha\models r+1\}$ is a complete list of distinct projective indecomposable $H_W(0)$-modules and $\{\C_\alpha:\alpha\models r+1\}$ is a complete list of distinct simple $H_W(0)$-modules.

\subsection{$0$-Hecke algebras of the symmetric groups} 
Now assume we are in type $A$, i.e. $W=\SS_n$ and $H_W(0)=H_n(0)$. The projective indecomposable $H_n(0)$-modules can be described in a combinatorial way \cite{HNT,KrobThibon} using ribbon tableaux. We see in the previous section that the ribbon tableaux of shape $\alpha$ are in bijection with the descent class of $\alpha$, hence in bijection with the basis of $\P_\alpha$ given above. The $H_n(0)$-action on $\P_\alpha$ agrees with the following $H_n(0)$-action on these ribbon tableaux: 
\begin{equation}\label{ActOnTableaux}
T_i\tau=\left\{\begin{array}{ll}
-\tau,& \textrm{if $i$ is in a higher row of $\tau$ than $i+1$},\\
0, &  \textrm{if $i$ is in the same row of $\tau$ as $i+1$},\\
s_i\tau, & \textrm{if $i$ is in a lower row of $\tau$ than $i+1$},
\end{array}\right.\end{equation}
where $\tau$ is a ribbon tableau of shape $\alpha$ and $s_i\tau$ is obtained from $\tau$ by swapping $i$ and $i+1$. This action gives rise to a directed version of the Hasse diagram of the interval $[w_0(\alpha),w_1(\alpha)]$ under the weak order. The top and bottom tableaux in this diagram correspond to $\C_\alpha={\rm top}\,(\P_\alpha )$ and $\C_{\overleftarrow\alpha}={\rm soc}\,(\P_\alpha)$. An example is given below for $\alpha=(1,2,1)$. 
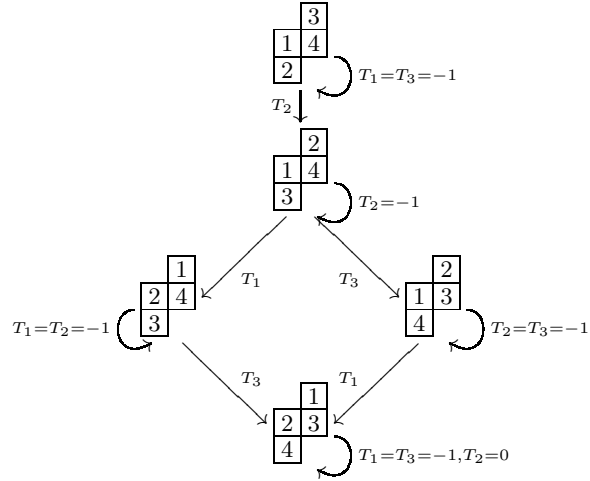
\[ 
{\small \xymatrix @R=12pt @C=24pt {
& {\young(:3,14,2)}  \ar@(ur,dr)[]^{T_1=T_3=-1} \ar@{->}[d]_-{T_2}\\
& {\young(:2,14,3)} \ar@(ur,dr)[]^{T_2=-1} \ar@{->}[ld]^{T_1} \ar@{->}[rd]_{T_3} \\
{\young(:1,24,3)} \ar@(ul,dl)[]_{T_1=T_2=-1} \ar@{->}[rd]^{T_3} & &
{\young(:2,13,4)} \ar@(ur,dr)[]^{T_2=T_3=-1} \ar@{->}[ld]_{T_1}\\
& {\young(:1,23,4)} \ar@(ur,dr)[]^{T_1=T_3=-1,T_2=0} 
}}  \] 

\subsection{Quasisymmetric and noncommutative symmetric functions}
Krob and Thibon~\cite{KrobThibon} provided a correspondence between representations of $H_n(0)$ and the dual Hopf algebras $\QSym$ of quasisymmetric functions and $\NSym$ of noncommutative symmetric functions.

The Hopf algebra $\QSym$ is a free $\ZZ$-module on the \emph{monomial quasisymmetric functions}
\[
M_\alpha:=\sum_{1\leq i_1<\cdots<i_\ell} x_{i_1}^{\alpha_1}\cdots x_{i_\ell}^{\alpha_\ell}
\]
for all compositions $\alpha=(\alpha_1,\ldots,\alpha_\ell)$. Another free $\ZZ$-basis consists of the \emph{fundamental quasisymmetric functions}
\[
F_\alpha:=\sum_{\alpha\cleq\beta} M_\beta = \sum_{\substack{ 1\leq i_1\leq\cdots\leq i_k\\ j\in D(\alpha)\Rightarrow i_j<i_{j+1} }} x_{i_1}\cdots x_{i_k}.
\]

The Hopf algebra $\NSym$ is the free associative algebra $\ZZ\langle\bh_1,\bh_2,\ldots\rangle$ where
\[
\bh_k:=\sum_{1\leq i_1\leq\cdots\leq i_k}\bx_{i_1}\cdots \bx_{i_k},\quad \forall\,k\geq1.
\]
A free $\ZZ$-basis for $\NSym$ consists of the \emph{complete homogeneous noncommutative symmetric functions} $\bh_\alpha:=\bh_{\alpha_1}\cdots \bh_{\alpha_\ell}$ for all compositions $\alpha=(\alpha_1,\ldots,\alpha_\ell)$. Another free $\ZZ$-basis consists of the \emph{noncommutative ribbon Schur functions}
\[
\bs_\alpha:=\sum_{\beta\cleq\alpha}(-1)^{\ell(\alpha)-\ell(\beta)}\bh_\beta
\]
for all compositions $\alpha$. The duality between $\QSym$ and $\NSym$ is given by $\langle M_\alpha, \bh_\beta\rangle = \langle F_\alpha, \bs_\beta\rangle :=\delta_{\alpha\beta}$.

The dual Hopf algebras $\QSym$ and $\NSym$ are also related to the self-dual Hopf algebra $\Sym$, the \emph{ring of symmetric functions}. A \emph{positive self-dual basis} for $\Sym$ consists of the \emph{Schur functions} $s_\lambda$ for all partitions $\lambda$. The definition of $s_\lambda$ is a special case of the \emph{skew Schur function} 
\[
s_{\lambda/\mu}:= \sum_{\tau\in \SSYT(\lambda/\mu)} x^\tau
\]
of a skew shape $\lambda/\mu$, where $x^\tau:=x_1^{d_1}x_2^{d_2}\cdots$ if $d_1,d_2,\ldots$ are the multiplicities of $1,2,\ldots$ in $\tau$. The commutative image of a noncommutative ribbon Schur function $\bs_\alpha$ is nothing but the ribbon Schur function $s_\alpha$. This gives a surjection $\NSym\twoheadrightarrow\Sym$. 

There is also a free $\ZZ$-basis for $\Sym$ consisting of the \emph{monomial symmetric functions} 
\[
m_\lambda:=\sum_{\lambda(\alpha)=\lambda} M_\alpha
\]
for all partitions $\lambda$. Here $\lambda(\alpha)$ is the unique partition obtained from the composition $\alpha$ by rearranging its parts. This gives an injection (actually an inclusion) $\Sym\hookrightarrow \QSym$. 

We will use the following expansion of a Schur function indexed by a partition $\lambda\vdash n$:
\begin{equation}\label{eq:Kostka}
s_\lambda = \sum_{\mu\vdash n} K_{\lambda\mu} m_\mu.
\end{equation}
Here $K_{\lambda\mu}$ is the Kostka number which counts all semistandard Young tableaux of shape $\lambda$ and type $\mu$. 

\subsection{Characteristic maps}
Let $A$ be an $\FF$-algebra and let $\mathcal C$ be a category of some finitely generated $A$-modules. The \emph{Grothendieck group of $\mathcal C$} is defined as the abelian group $F/R$, where $F$ is the free abelian group on the isomorphism classes $[M]$ of the $H_n(0)$-modules $M$ in $\mathcal C$, and $R$ is the subgroup of $F$ generated by the elements $[M]-[L]-[N]$ corresponding to all exact sequences $0\to L\to M\to N\to0$ of $A$-modules in $\mathcal C$. Note that every exact sequence of projective modules splits.

Denote by $G_0(\SS_n)$ the Grothendieck group of category of all finitely generated $\CC\SS_n$-modules. Then $G_0(\SS_n)$ is a free abelian group on the isomorphism classes of simple $\CC\SS_n$-modules $[S_\lambda]$ for all $\lambda\vdash n$. The tower of groups $\SS_\bullet: \SS_0\subset\SS_1\subset\SS_2\subset\cdots$ has a Grothendieck group 
\[
G_0(\SS_\bullet) := \bigoplus_{n\geq0} G_0(\SS_n).
\]
This is a self-dual Hopf algebra with product and coproduct given by induction and restriction of representations. The Frobenius characteristic map ch is defined by sending a simple $S_\lambda$ to the Schur function $s_\lambda$, giving a Hopf algebra isormorphism $G_0(\SS_\bullet)\cong\Sym$. 

There is an analogous result for $H_n(0)$-modules (over an arbitrary field $\FF$). The Grothendieck group of the category of all finitely generated $H_n(0)$-modules is denoted by $G_0(H_n(0))$, and the Grothendieck group of the category of finitely generated projective $H_n(0)$-modules is denoted by $K_0(H_n(0))$. We have 
\[
G_0(H_n(0))=\bigoplus_{\alpha\models n} \ZZ\cdot [\C_\alpha],\quad
K_0(H_n(0))=\bigoplus_{\alpha\models n} \ZZ\cdot [\P_\alpha].
\]
The Grothendieck groups of the tower of algebras $H_\bullet(0): H_0(0)\subset H_1(0)\subset H_2(0)\subset\cdots$ are defined as 
\[
G_0(H_\bullet(0)):=\bigoplus_{n\geq0}G_0(H_n(0)),\quad
K_0(H_\bullet(0)):=\bigoplus_{n\geq0}K_0(H_n(0)).
\]
These two Grothendieck groups are dual Hopf algebras with product and coproduct again given by induction and restriction of representations.  Krob and Thibon~\cite{KrobThibon} defined two Hopf algebra isomorphisms
\[
\mathrm{Ch}: G_0(H_\bullet(0))\cong\QSym,\quad \mathbf{ch}: K_0(H_\bullet(0))\cong \NSym.
\]
The \emph{quasisymmetric characteristic map} Ch sends a finitely generated $H_n(0)$-module $M$ with simple composition factors $\C_{\alpha^{(1)}},\ldots,\C_{\alpha^{(k)}}$ to 
\[ \mathrm{Ch}(M):=F_{\alpha^{(1)}}+\cdots+F_{\alpha^{(k)}}. \]
Similarly, the \emph{noncommutative characteristic map} $\mathbf{ch}$ sends a finitely generated projective $H_n(0)$-module $M=\P_{\alpha^{(1)}}\oplus\cdots\oplus\P_{\alpha^{(k)}}$ to 
\[
\mathbf{ch}(M):=\bs_{\alpha^{(1)}}+\cdots+\bs_{\alpha^{(k)}}.
\]
Krob and Thibon \cite{KrobThibon} also showed that $\mathrm{Ch}(\P_\alpha)=s_\alpha$, which is the commutative image of $\mathbf{ch}(P_\alpha)=\bs_\alpha$. Thus $\mathrm{Ch}(M)$ is symmetric whenever $M$ is a finitely generated projective $H_n(0)$-module, but not vice versa (e.g. $\C_{12}\oplus\C_{21}$ is nonprojective and has quasisymmetric characteristic equal to $s_{21}$).

If $M=H_n(0)v$ is cyclic then the \emph{length filtration}
\[ H_n(0)^{(\ell)}=\bigoplus_{\ell(w)\geq \ell}\FF\, T_w \]
induces a filtration of $H_n(0)$-modules $M^{(\ell)}=H_n(0)^{(\ell)}v$ for all $\ell\geq0$. This refines $\mathrm{Ch}(M)$ to a \emph{graded characteristic}
\[ \mathrm{Ch}_q(M)=\sum_{\ell\geq0} q^\ell\mathrm{Ch}\left(M^{(\ell)}/M^{(\ell+1)}\right). \]
Taking $M$ to be the regular representation of $H_n(0)$ we have
\[
\mathrm{Ch}_q(H_n(0))=\sum_{w\in \mathfrak S_n}q^{{\rm inv}(w)}
F_{D(w^{-1})}=\sum_\alpha  r_\alpha (q) F_\alpha 
\]
and taking a limit as $q\rightarrow 1$ gives
\[ \mathrm{Ch}(H_n(0))=\sum_\alpha  r_\alpha  F_\alpha. \]

If $M$ has another filtration by $H_n(0)$-modules $M_d$ for $d\geq0$, then one can look at the bifiltration of $H_n(0)$-modules
$M^{(\ell,d)}=M^{(\ell)}\cap M_d$ for $\ell,d\geq0$, and define the \emph{bigraded characteristic} to be
\[
\mathrm{Ch}_{q,t}(M)=\sum_{\ell,d\geq0} q^\ell t^d
\mathrm{Ch} \left(M^{(\ell,d)}/(M^{(\ell+1,d)}+M^{(\ell,d+1)})\right).
\]

If a projective $H_n(0)$-module $M=\P_{\alpha^{(1)}}\oplus\cdots\oplus\P_{\alpha^{(k)}}$ has a grading $\deg\P_{\alpha^{(i)}}=d_i$ then we can define the graded noncommutative characteristic of $M$ to be
\[
\mathbf{ch}_t(M):=t^{d_1}\bs_{\alpha^{(1)}}+\cdots+t^{d_k}\bs_{\alpha^{(k)}}.
\]

\section{Coinvariant algebra of $H_n(0)$}\label{SCoinvariantsA}

In this section we give interpretations of the ribbon and $q$-ribbon numbers by studying the $H_n(0)$-action on its coinvariant algebra.

The symmetric group $\mathfrak S_n$ acts on the polynomial ring $\FF[{\bx}]:=\FF[x_1,\ldots,x_n]$ over an arbitrary field $\FF$ by permuting the variables $x_1,\ldots,x_n$, and hence acts on the coinvariant algebra $\FF[\bx]/(\FF[\bx]^{\SS_n}_+)$ of $\SS_n$, where $(\FF[\bx]^{\SS_n}_+)$ is the ideal generated by symmetric polynomials of positive degree. We often identify the polynomials in $\FF[\bx]$ with their images in the coinvariant algebra of $\SS_n$ in this section.

For $i=1,\ldots,n-1$, let $s_i$ is the adjacent transposition $(i,i+1)$, and define the \emph{Demazure operators} 
\begin{equation}\label{eq:Demazure}
\pi_if:= \frac{x_if-x_{i+1}s_i(f)}{x_i-x_{i+1}},\quad \forall f\in\FF[\bx].
\end{equation}
It follows from this definition that
\begin{itemize}
\item $\pi_i$ satisfy the same relations as $T'_i$, i.e. the relations in (\ref{eq:T'i}),
\item $\deg(\pi_i f)=\deg(f)$ for all homogeneous polynomials $f\in\FF[\bx]$, 
\item $\pi_if=f$ if and only if $s_if=f$ for all $f\in\FF[\bx]$,
\item $\pi_i(fg)=f\pi_i(g)$ for all $f,g\in\FF[\bx]$ satisfying $\pi_if=f$.
\end{itemize}
Hence we have an $H_n(0)$-action on $\FF[\bx]$ by $T'_i\mapsto \pi_i$, or equivalently $T_i\mapsto \overline\pi_i:=\pi_i-1$, for $i=1,\ldots,n-1$. The operators $\pi_w$ and $\overline\pi_w$ are defined similarly as $T'_w$ and $T_w$ for all $w\in\SS_n$. This $H_n(0)$-action preserves the degrees, has the same invariants as $\SS_n$, and is $\FF[{\bx}]^{\mathfrak S_n}$-linear. Thus we have an  $H_n(0)$-action on the \emph{coinvariant algebra of $H_n(0)$}, which is defined as
\[
\FF[\bx]/ ( f\in\FF[\bx]: \deg f>0,\ \pi_i f = f,\ 1\leq i\leq n-1 )
\]
and coincides with the coinvariant algebra of $\SS_n$. 


To study the $H_n(0)$-action on its coinvariant algebra, we consider certain \emph{Demazure atoms} which behave nicely under the $H_n(0)$-action, i.e. the polynomials $\overline\pi_w x_{D(w)}$ for all $w\in\SS_n$. Here
\[
x_I:=\prod_{i\in I}x_1\cdots x_i.
\] 
for any $I\subseteq[n-1]$. See Mason~\cite{Mason} for more information on the Demazure atoms.

We will see in Lemma~\ref{DemazureAtoms} that the Demazure atoms mentioned above are closely related to the \emph{descent monomials}
\[
wx_{D(w)}=\prod_{i\in D(w)}x_{w(1)}\cdots x_{w(i)},\quad\forall w\in \mathfrak S_n.
\] 
It is well-known that the descent monomials form a basis for the coinvariant algebra of $\SS_n$; see e.g. Garsia~\cite{Garsia} and Steinberg~\cite{Steinberg}. Allen \cite{Allen} provided an elementary proof for this result, which we will adapt to the Demazure atoms $\overline\pi_w x_{D(w)}$. Thus we first recall Allen's proof below.

A \emph{weak composition} is a finite sequence of nonnegative integers. A \emph{partition} is a finite decreasing sequence of nonnegative integers, with zeros ignored sometimes. Every monomial in $\FF[\bx]$ can be written as $x^d=x_1^{d_1}\cdots x_n^{d_n}$ where $d=(d_1,\ldots,d_n)$ is a weak composition. Denote by $\lambda(d)$ the unique partition obtained from rearranging the weak composition $d$. Given two monomials $x^d$ and $x^e$, write $x^d\prec x^e$ or $d\prec e$ if $\lambda(d)<_L\lambda(e)$, and write $x^d<_{ts}x^e$ if (i) $\lambda(d)<_L\lambda(e)$ or (ii) $\lambda(d)=\lambda(e)$ and $d<_Le$, where ``$<_L$'' is the lexicographic order.

Given a weak composition $d=(d_1,\ldots,d_n)$, we have a permutation $\sigma(d)\in\SS_n$ obtained by labelling $d_1,\ldots,d_n$ from the largest to the smallest, breaking ties from left to right. Construct a weak composition $\gamma(d)$ from this labelling as follows. First replace the largest label with $0$, and recursively, if the label $t$ has been replaced with $s$, then replace $t-1$ with $s$ if it is to the left of $t$, or with $s+1$ otherwise. Let $\mu(d)=d-\gamma(d)$ be the
component-wise difference. For example,
\[
d=(3,1,3,0,2,0),\quad
\sigma(d)=(1,4,2,5,3,6),
\]
 \[
 \gamma(d)=(1,0,1,0,1,0),\quad
 \mu(d)=(2,1,2,0,1,0).
\]
The decomposition $d=\gamma(d)+\mu(d)$ is the usual \emph{$P$-partition encoding} of $d$ (see e.g. Stanley~\cite{Stanley}), and $x^{\gamma(d)}$ is the descent monomial of $\sigma(d)^{-1}$.  E.E. Allen \cite[Proposition~2.1]{Allen} showed that $w\mu(d)+\gamma(d)<_{ts} d$ for all $w\in\SS_n$ unless $w=1$, and thus
\begin{equation}\label{eq:Allen}
m_{\mu(d)}\cdot x^{\gamma(d)} = x^d +\sum_{x^e<_{ts}x^d} c_e x^e,\quad c_e\in\ZZ,
\end{equation}
where $m_{\mu(d)}$ is the monomial symmetric function corresponding to $\mu(d)$, i.e. the sum of the monomials in the $\SS_n$-orbit of $x^{\mu(d)}$. 
It follows that
\[
\left\{ m_\mu\cdot wx_{D(w)}:\mu=(\mu_1\geq\cdots\geq\mu_n\geq0),\ w\in\SS_n \right\}
\]
is triangularly related to the set of all monomials $x^d$, and thus an $\FF$-basis for $\FF[\bx]$. Therefore the descent monomials form an $\FF[\bx]^{\SS_n}$-basis for $\FF[\bx]$ and give an $\FF$-basis for $\FF[\bx]/(\FF[\bx]^{\SS_n}_+)$.

Now we investigate the relation between our Demazure atoms and the descent monomials. First observe that if $m$ is a monomial not containing $x_i$ and $x_{i+1}$, then 
\begin{equation}\label{pi}
\overline\pi_i(mx_i^ax_{i+1}^b)=\left\{\begin{array}{ll}
m(x_i^{a-1}x_{i+1}^{b+1}+x_i^{a-2}x_{i+1}^{b+2}\cdots +x_i^bx_{i+1}^a), & {\rm if}\ a>b,\\
0, &{\rm if}\ a=b,\\
-m(x_i^ax_{i+1}^b-x_i^{a+1}x_{i+1}^{b-1}-\cdots-x_i^{b-1}x_{i+1}^{a+1}), & {\rm if}\ a<b.
\end{array}\right.
\end{equation}

\begin{lemma}\label{DemazureAtoms}
Suppose that $\alpha$ is a composition of $n$ and $w$ is a permutation in $\SS_n$ with $D(w)\subseteq D(\alpha)$. Then
\[
\overline\pi_w x_{D(\alpha)} =w x_{D(\alpha)} +\sum_{x^d\prec x_{D(\alpha)}}c_dx^d, \quad c_d\in\ZZ.
\]
Moreover, $w x_{D(\alpha)}$ is a descent monomial if and only if $D(w)=D(\alpha)$.
\end{lemma}

\begin{proof}
We prove the first assertion by induction on $\ell(w)$. If $\ell(w)=0$ then we are done; otherwise we assume $w=s_ju$ for some $j\in[n-1]$ and some $u\in\SS_n$ with $\ell(u)<\ell(w)$. Since $D(u)\subseteq D(w)\subseteq D(\alpha)$, one has
\begin{equation}\label{eq1}
\overline\pi_u x_{D(\alpha)} =ux_{D(\alpha)} +\sum_{x^d\prec x_{D(\alpha)} }c_dx^d, \quad c_d\in\ZZ.
\end{equation}
It follows from (\ref{pi}) that
\begin{equation}\label{eq2}
\overline\pi_j(x^d)=\sum_{x^e\preceq x^d}a_e x^e, \quad a_e\in\ZZ.
\end{equation}
Observe that the degree of $x_k$ in $ux_{D(\alpha)}$ is 
\[ r_k=\#\{i\in D(\alpha): u^{-1}(k)\leq i\}. \]
Since $\ell(s_ju)>\ell(u)$, we have $u^{-1}(j)<u^{-1}(j+1)$ and thus $r_j\geq r_{j+1}$. Since $(s_ju)^{-1}(j+1)<(s_ju)^{-1}(j)$, there exists an $i\in D(s_ju)\subseteq D(\alpha)$ such that 
\[
u^{-1}(j)=(s_ju)^{-1}(j+1)\leq i<(s_ju)^{-1}(j)=u^{-1}(j+1).
\]
Thus $r_j>r_{j+1}$. It follows from (\ref{pi}) that
\begin{equation}\label{eq3}
\overline\pi_j(ux_{D(\alpha)} )=s_jux_{D(\alpha)} +\sum_{x^e\prec x_{D(\alpha)} } b_ex^e, \quad b_e\in\ZZ.
\end{equation}
Combining (\ref{eq1}), (\ref{eq2}), and (\ref{eq3}) proves the first assertion.

If $D(w)=D(\alpha)$ then $wx_{D(\alpha)}$ is the descent monomial of $w$. Conversely, assume $wx_{D(\alpha)}$ equals the descent monomial of some $u\in W$, i.e.
\[
\prod_{i\in D(\alpha)}x_{w(1)}\cdots x_{w(i)}=\prod_{j\in D(u)} x_{u(1)}\cdots x_{u(j)}.
\]
Let $D(\alpha)=\{i_1,\ldots,i_k\}$ and $D(u)=\{j_1,\ldots,j_t\}$. Comparing the variables absent on both sides of the above equality, one sees that $i_k=j_t$ and $w(i)=u(i)$ for $i=i_k+1,\ldots,n$. Repeat this argument for the variables appearing exactly $m$ times, $m=1,2,\ldots,$ one sees that $D(\alpha)=D(u)$ and $w=u$.
\end{proof}

\begin{remark}\label{rem:atom}
Using the combinatorial formula by Mason~\cite{Mason} for the Demazure atoms, one can check that $\overline\pi_{w_0(\alpha)}x_{D(\alpha)}$ and $\overline\pi_{w_1(\alpha)}x_{D(\alpha)}$ are precisely the descent monomials of $w_0(\alpha)$ and $w_1(\alpha)$.
\end{remark}

\begin{lemma}\label{Straightening}
For any weak composition $d=(d_1,\ldots,d_n)$, let $\sigma=\sigma(d)$, $\gamma=\gamma(d)$, $\mu=\mu(d)$. If $c_\beta\in\ZZ$ for all $\beta\prec\gamma$ then 
\[
m_\mu\cdot \left(x^\gamma+ \sum_{\beta\prec \gamma} c_\beta x^\beta\right) = x^d +\sum_{x^e<_{ts}x^d} b_e x^e,\quad b_e\in\ZZ.
\]
\end{lemma}

\begin{proof}
Since we already have (\ref{eq:Allen}), it suffices to show that $w\mu+\beta\prec d$ for all permutations $w$ in $\mathfrak S_n$ and all $\beta\prec \gamma$. Given a weak composition $\alpha$, let $\alpha_i$ be its $i$-th part. Since $\sigma\mu$ and $\sigma\gamma$ are both weakly decreasing, one has $\lambda(\mu)_i+\lambda(\gamma)_i=\lambda(d)_i$ for all $i=1,\ldots,n$. Since $\beta\prec\gamma$, there exists a unique integer $k$ such that $\lambda(\beta)_i=\lambda(\gamma)_i$ for $i=1,\ldots,k-1$, and $\lambda(\beta)_k<\lambda(\gamma)_k.$ Then for all $i\in[k-1]$, 
\begin{eqnarray*}
\lambda(w\mu+\beta)_i
&\leq& \lambda(\mu)_i+\lambda(\beta)_i \\
&=& \lambda(\mu)_i+\lambda(\gamma)_i\\
&=&\lambda(d)_i
\end{eqnarray*}
and
\begin{eqnarray*}
\lambda(w\mu+\beta)_k
&\leq& \lambda(\mu)_k+\lambda(\beta)_k\\
&<& \lambda(\mu)_k+\lambda(\gamma)_k \\
&=& \lambda(d)_k.
\end{eqnarray*}
Therefore $w\mu+\beta\prec d$ and we are done.
\end{proof}

\begin{lemma}\label{lem:basis}
The coinvariant algebra $\FF[\bx]/(\FF[\bx]^{\SS_n}_+)$ has a basis given by $\{f_w:w\in\SS_n\}$, if 
\[
f_w = wx_{D(w)}+\sum_{x^d\prec x_{D(\alpha)}}c_dx^d,\quad c_d\in\FF,\quad \forall w\in\SS_n.
\]
\end{lemma}

\begin{proof}
Given a weak composition $d=(d_1,\ldots,d_n)$, let $\gamma=\gamma(d)$, $\mu=\mu(d)$, and $\sigma=\sigma(d)$. Then $x^\gamma$ is the descent monomial of $\sigma^{-1}$. By Lemma~\ref{Straightening},
\[
m_\mu f_{\sigma^{-1}}=x^d+\sum_{x^e<_{ts}x^d}b_ex^e.
\]
Hence $\{m_\mu f_w:\mu=(\mu_1\geq\cdots\geq\mu_n\geq0),\ w\in\SS_n\}$ is triangularly related to the set of all monomials $x^d$, and thus a basis for $\FF[\bx]$. It follows that  $\{f_w:w\in\mathfrak S_n\}$ is an $\FF[{\bx}]^{\mathfrak S_n}$-basis for $\FF[{\bx}]$ and gives an $\FF$-basis for $\FF[{\bx}]/(\FF[{\bx}]^{\mathfrak S_n}_+)$. 
\end{proof}

\begin{theorem}\label{CoinvariantsA}
The coinvariant algebra $\FF[{\bx}]/(\FF[\bx]^{\SS_n}_+)$ has an $\FF$-basis $\left\{\overline\pi_w x_{D(w)} :w\in\mathfrak S_n \right\}$ and decomposes as
\[
\FF[{\bx}]/(\FF[\bx]^{\SS_n}_+)=\bigoplus_{\alpha\models n} H_n(0)\cdot\overline\pi_{w_0(\alpha)} x_{D(\alpha)}
\]
where $H_n(0)\cdot\overline\pi_{w_0(\alpha)} x_{D(\alpha)}$ has a basis $\left\{\overline\pi_w x_{D(\alpha)} :w\in[w_0(\alpha),w_1(\alpha)]\right\}$, and is isomorphic to the projective indecomposable $H_n(0)$-module $\P_\alpha$, for all $\alpha\models n$. Consequently, $\FF[{\bx}]/(\FF[{\bx}]^{\mathfrak S_n}_+)$ is isomorphic to the regular representation of $H_n(0)$.
\end{theorem}

\begin{proof}
By Lemma~\ref{DemazureAtoms} and Lemma~\ref{lem:basis}, $\{\overline\pi_w x_{D(w)}:w\in\mathfrak S_n\}$ is a basis for $\FF[{\bx}]/(\FF[\bx]^{\SS_n}_+)$. For any permutation $u$ in $\mathfrak S_n$, one sees from the relations of $\overline\pi_i$ that $\overline\pi_u\overline\pi_{w_0(\alpha)}=\pm\overline\pi_w$ for some $w\geq w_0(\alpha)$ in the left weak order, which implies $D(w)\supseteq D(\alpha)$. If there exists $j\in D(w)\setminus D(\alpha)$, then 
\[
\overline\pi_w x_{D(\alpha)} = \overline\pi_{ws_j}\overline\pi_jx_{D(\alpha)} =0.
\]
Hence $H_n(0)\cdot\overline\pi_{w_0(\alpha)} x_{D(\alpha)}$ is spanned by $\{\overline\pi_w x_{D(\alpha)} :w\in[w_0(\alpha),w_1(\alpha)]\}$, which must be an $\FF$-basis since it is a subset of a linearly independent set. Sending $\overline\pi_w x_{D(\alpha)} $ to $T_wT'_{w_0(\alpha^c)}$ for all $w\in[w_0(\alpha),w_1(\alpha)]$ gives an isomorphism between $ H_n(0)\cdot\overline\pi_{w_0(\alpha)} x_{D(\alpha)}$ and $\P_\alpha$.
\end{proof}

\begin{remark}
(i) This theorem and its proof are valid when $\FF$ is replaced with $\ZZ$. 

\noindent(ii) By Remark~\ref{rem:atom}, the cyclic generators $\overline\pi_{w_0(\alpha)} x_{D(\alpha)}$ for the indecomposable summands of the coinvariant algebra are precisely the descent monomials $w_0(\alpha) x_{D(\alpha)}$.
\end{remark}

\begin{theorem}[Lusztig-Stanley]
The graded Frobenius characteristic of the coinvariant algebra is
\[
\mathrm{ch}_t\left(\CC[\bx]/(\CC[\bx]^{\SS_n}_+)\right) =  \sum_{\lambda\vdash n} \sum_{\tau\in\mathrm{SYT}(\lambda)} t^{\maj(\tau)} s_\lambda \xlongequal{(*)} \widetilde H_{1^n}(x;t)
\]
where $\widetilde H_{1^n}(x;t)$ is the modified Hall-Littlewood symmetric function of the partition $1^n$.
\end{theorem}

\begin{remark}
The equality ($*$) is a special case of Theorem~\ref{thm:Springer}. One can also see it by using the charge formula of Lascoux and Sch\"utzenberger~\cite{LS}.
\end{remark}

We have an analogues result for the $H_n(0)$-action on the coinvariant algebra.

\begin{corollary}\label{cor:coinv}
(i) The bigraded characteristic of the coinvariant algebra is
\[
\mathrm{Ch}_{q,t}\left(\FF[{\bx}]/(\FF[{\bx}]^{\mathfrak S_n}_+)
\right) = \sum_{w\in\mathfrak S_n}t^{{\rm maj}(w)}q^{{\rm
inv}(w)}F_{D(w^{-1})}
\]
(ii) The (degree) graded quasisymmetric characteristic of the coinvariant algebra is 
\[
\mathrm{Ch}_t\left(\FF[{\bx}]/(\FF[{\bx}]^{\mathfrak S_n}_+)
\right) = \sum_{\alpha\models n} r_\alpha (t) F_\alpha = \sum_{\alpha\models n} \sum_{\tau\in\mathrm{SYT}(\alpha)}t^{\maj(\tau)} F_\alpha.
\]
(iii) The quasisymmetric function in (ii) is actually symmetric and equals 
\[
\sum_{\lambda\vdash n}\qbin{n}{\lambda}{t} m_\lambda = \sum_{\lambda\vdash n} \sum_{\tau\in\mathrm{SYT}(\lambda)} t^{\maj(\tau)} s_\lambda = \sum_{\lambda\vdash n} t^{n(\lambda)}\frac{[n]!_t} {\prod_{u\in\lambda}[h_u]_t} s_{\lambda} = \widetilde H_{1^n}(x;t)
\]
where $h_u$ is the hook length of the box $u$ in the Ferrers diagram of $\lambda$ and $n(\lambda):=\lambda_2+2\lambda_3+3\lambda_4+\cdots$. 
\end{corollary}

\begin{proof}
Given a composition $\alpha$ of $n$, the $H_n(0)$-module $H_n(0)\cdot\overline\pi_{w_0(\alpha)} x_{D(\alpha)}$ consists of homogeneous elements of degree ${\rm maj}(w)$ for any $w\in[w_0(\alpha),w_1(\alpha)]$. Hence Theorem~\ref{CoinvariantsA} implies (i). It follows from (i) that the degree graded multiplicity of a simple $H_n(0)$-module $\C_\alpha$ in the coinvariant algebra is
\[
r_\alpha (t) = \sum_{ \substack{ w\in\SS_n: \\ D(w^{-1})=D(\alpha) }} t^{\maj(w)} = \sum_{ \substack{ w\in\SS_n: \\ D(w)=D(\alpha) }} t^{\maj(w^{-1})} = \sum_{\tau\in\mathrm{SYT}(\alpha)}t^{\maj(\tau)}.
\] 
Then 
\[
\mathrm{Ch}_t\left(\FF[{\bx}]/(\FF[{\bx}]^{\mathfrak S_n}_+)\right) = \sum_{\alpha\models n} r_\alpha(t) F_\alpha = \sum_{\alpha\models n} \qbin{n}{\alpha}{t} M_\alpha = \sum_{\lambda\vdash n}\qbin{n}{\lambda}{t} m_\lambda \in\Sym[t].
\]
 
Given a partition $\mu$ of $n$, we have
\[
\qbin{n}{\mu}{t} = \sum_{ \substack{ w\in\SS_n: \\ D(w)\subseteq D(\mu) }} t^{\maj(w^{-1})}
= \sum_{w\in \SS(\mu)} t^{\maj(w^{-1})} 
\]
where $\SS(\mu)$ is the set of all permutations of the multiset of type $\mu$. For example, $w=3561247$ corresponds to
\[
\begin{pmatrix}
1 & 1 & 1 & 2 & 2 & 3 & 3 \\
3 & 5 & 6 & 1 & 2 & 4 & 7
\end{pmatrix}\in\SS(332).
\]
Applying RSK to $w\in\SS(\mu)$ gives a pair $(P,Q)$ of Young tableaux $P$ and $Q$ of the same shape (say $\lambda$), where $P$ is standard, and $Q$ is semistandard of type $\mu$. It is well-known that the descents of $w^{-1}$ are precisely the descents of $P$; see Sch\"utzenberger~\cite{DesRSK}. Hence
\begin{eqnarray*}
\mathrm{Ch}_t\left(\FF[{\bx}]/(\FF[{\bx}]^{\mathfrak S_n}_+)\right)  & = &\sum_{\lambda\vdash n} \sum_{P\in{\rm SYT(\lambda)}} t^{\maj(P)} \sum_{\mu\vdash n} K_{\lambda\mu} m_\mu \\
& = & \sum_{\lambda\vdash n} t^{n(\lambda)}\frac{[n]!_t}{\prod_{u\in\lambda}[h_u]_t} s_{\lambda}.
\end{eqnarray*}
Here the last equality follows from the the $q$-hook length formula and (\ref{eq:Kostka}).
\end{proof}



\section{Coinvariant algebra of Weyl groups}\label{SCoinvariantsW}

The results in the previous section can be generalized to the action of the $0$-Hecke algebra of a Weyl group $W$ on the Laurent ring $\FF[\Lambda]$ of the weight lattice $\Lambda$ of $W$. The readers are referred to Humphreys~\cite{Humphreys} for details on the Weyl groups and weight theory.

Demazure's character formula \cite{Demazure} expresses the character of the highest weight modules over a semisimple 
Lie algebra using the \emph{Demazure operators} $\pi_i$ on the group
ring $\FF[\Lambda]$ of the weight lattice $\Lambda$. Write $e^\lambda$ for the element in $\FF[\Lambda]$ corresponding to
the weight $\lambda\in\Lambda$. Suppose that $\gamma_1,\ldots,\gamma_r$ are the simple roots\footnote{We use $\gamma$ to denote roots because $\alpha$ is used for compositions throughout this paper.}, $s_1,\ldots,s_r$ are the simple reflections, and $\lambda_1,\ldots,\lambda_r$ are the fundamental weights. Then
\[
\FF[\Lambda]=\FF[z_1,\ldots,z_r,z_1^{-1},\ldots,z_r^{-1}]
\]
where $z_i=e^{\lambda_i}$. The \emph{Demazure operators} are defined by 
\[ \pi_i=\frac{f-e^{-\gamma_i}s_i(f)}{1-e^{-\gamma_i}}, \quad \forall f\in\FF[\Lambda]. \]
It follows that \begin{equation}\label{DemaWt}
\pi_i(e^\lambda)=\left\{\begin{array}{ll}
e^\lambda+e^{\lambda-\gamma_i}+\cdots+e^{s_i\lambda}, & {\rm if}\
\langle \lambda,\gamma_i \rangle\geq 0,\\
0, & {\rm if}\ \langle \lambda,\gamma_i \rangle= -1,\\
-e^{\lambda+\gamma_i}-\cdots-e^{s_i\lambda-\gamma_i}, & {\rm if}\
\langle \lambda,\gamma_i \rangle <-1.
\end{array}\right.
\end{equation}
Here $\langle \lambda,\gamma_i \rangle=2(\lambda,\gamma_i)/(\gamma_i,\gamma_i)$ with $(-,-)$ being the standard inner product. See, for example, Kumar \cite{Kumar}. The Demazure operators satisfy $s_i\pi_i=\pi_i$, $\pi_i^2=\pi_i$, and the braid relations \cite[\S5.5]{Demazure}. Hence the $0$-Hecke algebra $H_W(0)$ of the Weyl group $W$ acts on $\FF[\Lambda]$ by sending $T_i$ to $\overline\pi_i=\pi_i-1$. If $w=s_{i_1}\cdots s_{i_k}$ is a reduced expression then $T_w$ acts by $\overline\pi_w=\overline\pi_{i_1}\cdots\overline\pi_{i_k}$. It is clear that $\pi_if=f$ if and only if $s_if=s_if$ for all $f\in\FF[\Lambda]$.

Using the Stanley-Reisner ring of the Coxeter complex of $W$, Garsia and Stanton \cite{GarsiaStanton} showed that
\[
\FF[\Lambda]^W=\FF[a_1,\ldots,a_r]
\]
where
\[ a_i=\sum_{w\in W/W_{i^c}}e^{w\lambda_i},\quad(i^c=[r]\setminus\{i\})
\]
and $\FF[\Lambda]$ has a free basis over $\FF[\Lambda]^W$, which consists of the \emph{descent monomials}
\[ z_w:=\prod_{i\in D(w)} e^{w\lambda_i},\quad \forall w\in W. \]
See also Steinberg \cite{Steinberg}. If we write $\lambda_I=\sum_{i\in I}\lambda_i$ for all subsets $I\subseteq[r]$, then $z_w=we^{\lambda_{D(w)}}$. The basis $\{z_w:w\in W\}$ induces an $\FF$-basis for $\FF[\Lambda]/(a_1,\ldots,a_r).$ The $H_W(0)$-action on $\FF[\Lambda]$ is $\FF[\Lambda]^W$-linear, hence inducing an $H_W(0)$-action on $\FF[\Lambda]/(a_1,\ldots,a_r)$. 

We order the weights by $\lambda\leq\mu$ if $\mu-\lambda$ is a nonnegative linear combination of simple roots. Every monomial is $\FF[\Lambda]$ is of the form $m=e^\lambda$ for some weight $\lambda$. By Humphreys~\cite{Humphreys}, there exists a unique dominant weight $\mu$ such that $\mu=w\lambda$ for some $w$ in $W$, and we have $\lambda\leq\mu$. Write $[m]_+=[\lambda]_+:=\mu$ and call this dominant weight $\mu$  the \emph{shape} of the monomial $m$ or the weight $\lambda$.

For every monomial $m$ of shape $\lambda$, Garsia and Stanton \cite[proof of Theorem 9.4]{GarsiaStanton} showed that
\begin{equation}\label{DesAlgorithm}
m-\sum_{ \substack{d\in\ZZ^r,\ w\in W: \\ \lambda_d+\lambda_{D(w)} = \lambda }} c_{d,w}
a_1^{d_1}\cdots a_r^{d_r} z_w
\end{equation}
is a linear combination of monomials whose shape is strictly less than $\lambda$, where $\lambda_d=d_1\lambda_1+\cdots+d_r\lambda_r$ and $c_{d,w}\in\ZZ$. It follows from induction that the descent monomials $z_w$ form an $\FF[\Lambda]^W$-basis for $\FF[\Lambda]$.

\begin{lemma}\label{shape}
Suppose that $\gamma$ is a simple root and $\lambda$ is a weight such that $\langle \lambda,\gamma \rangle\geq0$. If $0\leq k\leq \langle \lambda,\gamma \rangle$ then $[\lambda-k\gamma]_+\leq[\lambda]_+$, and the equality holds if and only if $k=0$ or $\langle \lambda,\gamma \rangle$.
\end{lemma}

\begin{proof}
Let $\mu=\lambda-k\gamma$. If $k=0$ or $\langle \lambda,\gamma \rangle$, then $\mu=\lambda$ or
$s_\gamma\lambda$, and thus $[\mu]_+=[\lambda]_+$ in either case.

Now assume $0<k<\langle\lambda,\gamma \rangle$, and let $w\lambda$ and $u\mu$ be dominant for some $w$ and $u$ in $W$.

If $u\gamma>0$ then $u\mu=uw^{-1}(w\lambda)-ku\gamma<uw^{-1}(w\lambda)\leq w\lambda$.

If $u\gamma<0$ then
\begin{eqnarray*}
u\mu&=&us_\gamma\cdot s_\gamma\lambda-ku\gamma\\
&=&us_\gamma(\lambda-\langle\lambda,\gamma \rangle\gamma)-ku\gamma\\
&=&us_\gamma\lambda+(\langle\lambda,\gamma \rangle-k)u\gamma\\
&<&us_\gamma\lambda\leq w\lambda.
\end{eqnarray*}
Thus we are done.
\end{proof}

\begin{lemma}\label{DesPolyExpansion}
Given a composition $\alpha$ of $r+1$, let $\lambda_\alpha=\lambda_{D(\alpha)}$ and $z_\alpha =e^{\lambda_\alpha}$. If $w\in W$ has $D(w)\subseteq D(\alpha)$, then
\[
\overline\pi_w z_\alpha =e^{w\lambda_\alpha }+ \sum_{[\lambda]_+<\lambda_\alpha }
c_\lambda e^\lambda, \quad c_\lambda\in\ZZ.
\]
Moreover, $e^{w\lambda_\alpha }$ is a descent
monomial if and only if $D(w)=D(\alpha)$.
\end{lemma}

\begin{proof}
We prove the first assertion by induction on $\ell(w)$.
If $\ell(w)=0$ then we are done; otherwise $w=s_ju$
for some $j\in[r]$ and for some $u$ with $\ell(u)<\ell(w)$.
Since $D(u)\subseteq D(w)\subseteq D(\alpha)$, one has
\[
\overline\pi_u z_\alpha =e^{u\lambda_\alpha }+ \sum_{[\lambda]_+<\lambda_\alpha }
c_\lambda e^\lambda,\quad c_\lambda\in\ZZ.
\]
Applying Lemma~\ref{shape} to (\ref{DemaWt}) (if the simple root $\gamma_j$ satisfies $\langle\lambda,\gamma_j\rangle\leq0$ then
$\langle s_j\lambda,\gamma_j\rangle\geq0$), one sees that 
\[
\overline\pi_j(e^\lambda)=\sum_{[\mu]_+\leq[\lambda]_+}a_\mu e^\mu,
\quad a_\mu\in\ZZ.
\]
If we can show $\langle u\lambda_\alpha ,\gamma_j \rangle>0$, then
applying Lemma~\ref{shape} to the first case of (\ref{DemaWt}) one
has
\[
\overline\pi_j e^{u\lambda_\alpha }=e^{s_ju\lambda_\alpha }+
\sum_{[\mu]_+<\lambda_\alpha } b_\mu e^\mu,\quad b_\mu\in\ZZ.
\]
Combining these equations one obtains
\[
\overline\pi_w z_\alpha =e^{w\lambda_\alpha }+ \sum_{[\mu]_+<\lambda_\alpha } b_\mu
e^\mu +\sum_{[\lambda]_+<\lambda_\alpha }c_\lambda
\sum_{[\mu]_+\leq[\lambda]_+}a_\mu e^\mu,
\]
which gives the desired result.

Now we prove $\langle u\lambda_\alpha ,\gamma_j \rangle>0$. In fact, since $\ell(s_ju)>\ell(u)$, one has $u^{-1}(\gamma_j)>0$, i.e.
\[ u^{-1}(\gamma_j)=\sum_{i=1}^r m_i\gamma_i \]
for some nonnegative integers $m_i$. Applying $s_ju$ to both sides one gets
\[ 0>-\gamma_j=\sum_{i=1}^r m_is_ju(\gamma_i). \]
By the hypothesis $D(s_ju)\subseteq D(\alpha)$, if $i\notin D(\alpha)$ then $s_ju(\gamma_i)>0$. This forces $m_i>0$ for some $i\in D(\alpha)$, and thus
\[
\langle u\lambda_\alpha ,\gamma_j \rangle=\langle
\lambda_\alpha ,u^{-1}\gamma_j \rangle=\sum_{i\in D(\alpha)}m_i>0.
\]

Finally we consider when $e^{w\lambda_\alpha }$ is a descent monomial. If $D(w)=D(\alpha)$ then it is just the descent monomial of $w$. Conversely, if it is a descent monomial of some $u\in W$ then $w\lambda_\alpha =u\lambda_{D(u)}$. Since $\lambda_\alpha $ and $\lambda_{D(u)}$ are both in the fundamental Weyl chamber, the above equality implies that $\lambda_\alpha =\lambda_{D(u)}$ and $u^{-1}w$ is a product of simple reflections which all fix $\lambda_\alpha $ (\cite[Lemma 10.3B]{Humphreys}), {\it i.e.}
\[
w=us_{j_1}\cdots s_{j_k},\quad j_1,\ldots,j_k\notin D(u)=D(\alpha).
\]
Since $D(w)\subseteq D(\alpha)=D(u)$, none of $s_{j_1},\ldots,s_{j_k}$ is a descent of $w$, and thus it follows from the deletion property of $W$ that $w$ is a subword of some reduced expression of $u$, i.e. $w\leq u$ in Bruhat order. Similarly, it follows from $u=ws_{j_k}\cdots s_{j_1}$ that $u\leq w$ in Bruhat order. Thus $u=w$.
\end{proof}

\begin{theorem}\label{WeylCoinvariants}
The coinvariant algebra $\FF[\Lambda]/(a_1,\ldots,a_r)$ has an $\FF$-basis $\left\{ \overline\pi_w e^{\lambda_{D(w)}}:w\in W \right\}$ and decomposes as
\[
\FF[\Lambda]/(a_1,\ldots,a_r)=\bigoplus_{\alpha\models r+1} H_W(0)\cdot\overline\pi_{w_0(\alpha)} z_\alpha
\]
where each direct summand $H_W(0)\cdot\overline\pi_{w_0(\alpha)} z_\alpha$ has an $\FF$-basis \[
\left\{\overline\pi_w z_\alpha :w\in[w_0(\alpha),w_1(\alpha)]\right\}
\]
and is isomorphic to the projective indecomposable $H_W(0)$-module $\P_\alpha$. Consequently, $\FF[\Lambda]/(a_1,\ldots,a_r)$ is isomorphic to the regular representation of $H_W(0)$.
\end{theorem}

\begin{proof}
By Lemma \ref{DesPolyExpansion}, if one replaces the descent
monomial $z_w$ with the Demazure atom $\overline\pi_w e^{\lambda_{D(w)}}$ in
(\ref{DesAlgorithm}), the extra terms produced are of the form 
\[
c_{d,w}a_1^{d_1}\cdots a_r^{d_r}e^\mu
\]
where $d=(d_1,\ldots,d_r)$ and $\mu$ are weak compositions, and  $w$ is an element in $W$, satisfying $\lambda_d+\lambda_{D(w)}=\lambda$ and $[\mu]<\lambda_{D(w)}$. By the definition of $a_1,\ldots,a_r$, one expands each term above as a linear combination of the monomials
\[
e^{d_1w_1\lambda_1}\cdots e^{d_r w_r \lambda_r}e^\mu,\quad w_i\in W/W_{i^c}.
\]
There exists $w\in W$ such that
\begin{eqnarray*}
&&\left[\mu+ d_1w_1\lambda_1+\cdots+d_r w_r \lambda_r \right] \\
&=& w(\mu+d_1w_1\lambda_1+\cdots+d_r w_r \lambda_r ) \\
&\leq& [\mu]+ d_1\lambda_1+\cdots+d_r \lambda_r \\
&<& \lambda_{D(w)}+\lambda_d=\lambda.
\end{eqnarray*}
By induction on the shapes of the monomials, one shows that the Demazure atoms $\overline\pi_w\lambda_{D(w)}$ for all $w$ in $W$ form an $\FF[\Lambda]^W$-basis for $\FF[\Lambda]$, giving an $\FF$-basis for the coinvariant algebra $\FF[\Lambda]/(a_1,\ldots,a_r)$.

Let $\alpha\models r+1$. For any $u$ in $W$, using $\overline\pi_i^2=-\overline\pi_i$ one shows by induction that $\overline\pi_u\overline\pi_{w_0(\alpha)}=\overline\pi_w$ for some $w\geq w_0(\alpha)$ in the (left) weak order, which implies $D(w)\supseteq D(\alpha)$. On the other hand, if there exists $j\in D(w)\setminus D(\alpha)$, then $\overline\pi_w z_\alpha =0$
since $\overline\pi_j z_\alpha=0$ by (\ref{DemaWt}). Hence $H_W(0)\cdot\overline\pi_{w_0(\alpha)} z_\alpha$ has a basis $\{\overline\pi_w z_\alpha :w\in[w_0(\alpha),w_1(\alpha)]\}$, and is isomorphic to $\P_\alpha$ via $\overline\pi_w z_\alpha \mapsto T_wT'_{w_0(\alpha^c)}$ for all $w\in[w_0(\alpha),w_1(\alpha)]$.
\end{proof}

\begin{remark}
Garsia and Stanton \cite{GarsiaStanton} pointed out a way to reduce the descent monomials in $\FF[\Lambda]$ to the descent monomials in $\FF[{\bx}]$ for type $A$. However, it does not give Theorem \ref{CoinvariantsA} directly from Theorem \ref{WeylCoinvariants}; instead, one should consider the Demazure operators on $\FF[X(T)]$ where $X(T)$ is the character group of the subgroup $T$ of diagonal matrices in $GL(n,\FF)$.
\end{remark}

\section{Flag varieties}\label{SBG}

In this section we assume $\FF$ is a field of characteristic $p>0$ and study the action of the $0$-Hecke algebras on the (complete) flag varieties. Let $G$ be a finite group of Lie type over a finite field $\FF_q$ of characteristic $p$, with Borel subgroup $B$ and Weyl group $W$. 
Assume that $W$ is generated by simple reflections $s_1,\ldots,s_r$. Every composition $\alpha$ of $r+1$ corresponds to a \emph{parabolic subgroup} $P_\alpha:=BW_{D(\alpha)^c}B$ of $G$. The \emph{partial flag variety} $1^G_{P_\alpha}$ is the induction of the trivial representation of $P_\alpha$ to $G$, or in other words, the $\FF$-span of all right $P_\alpha$-cosets in $G$. Taking $\alpha=1^{r+1}$ we have the \emph{(complete) flag variety} $1_B^G$.  

For type $A$, one has $G=GL(n,\mathbb F_q)$, and if $\alpha=(\alpha_1,\ldots,\alpha_\ell)$ is a composition of $n$, then $P_\alpha$ is the group of all block upper triangular matrices with invertible diagonal blocks of sizes $\alpha_1,\ldots,\alpha_\ell$. Using the action of $G$ on the vector space $V=\FF^n$, one can identify $1_{P_\alpha}^G$ with the $\FF$-span of all partial flags of subspaces $0\subset V_1\subset\cdots\subset V_\ell=V$ satisfying $\dim V_i = \alpha_i$ for $i=1,\ldots,\ell$; in particular, $1_B^G$ is the $\FF$-span of all complete flags of $V$. 

\subsection{$0$-Hecke algebra action on $1_B^G$}
Given a subset $H\subseteq G$, let $\overline H=\sum_{h\in H}h$ in $\ZZ G$. Then $1_B^G=\overline B\cdot \FF G$. By work of  Kuhn \cite{Kuhn}, the endomorphism ring ${\rm End}_{\ZZ G}(\overline B\cdot\ZZ G)$ has a basis $\{f_w:w\in W\}$, with $f_w$ given by
\[
f_w(\overline B)=\overline {BwB}=\overline U_w w\overline B
\]
where $U_w$ be the product of the root subgroups of those positive roots which are sent to negative roots by $w^{-1}$ (see e.g.~\cite[Proposition 1.7]{DigneMichel}). The endomorphism ring ${\rm End}_{\ZZ G}(\overline B\cdot\ZZ G)$ is isomorphic to the Hecke algebra $H_W(q)$ of $W$ with parameter $q=|U_{s_i}|$, since the relations satisfied by $\{f_w:w\in W\}$ are the same as those satisfied by the standard basis for $H_W(q)$. Working over a field $\FF$ of characteristic $p$ in which $q=0$, we obtain a $G$-equivariant action of the $0$-Hecke algebra $H_W(0)$ on $1_B^G$ by 
\[
T_w(\overline Bg):=\overline{BwB}g,\quad \forall g\in G,\ \forall w\in W.
\]
We will use left cosets in the next subsection, and in that case there is a similar right $H_W(0)$-action.

Given a finite dimensional filtered $H_W(0)$-module $Q$ and a composition $\alpha$ of $r+1$, define $Q_\alpha$ to be the $\FF$-subspace of the elements in $Q$ that are annihilated by $T_j$ for all $j\notin D(\alpha)$, i.e. 
\[
Q_\alpha:=\bigcap_{j\in D(\alpha)^c} \ker T_j.
\]
The next lemma gives the simple composition factors of $Q$ by inclusion-exclusion. We do not need any (nontrivial) filtration for $Q$ in this section, but we will need it in the next section.


\begin{lemma}\label{Ft}
Given a finite dimensional filtered $H_W(0)$-module $Q$ and a composition $\alpha$ of $r+1$,  the graded multiplicity of the simple $H_W(0)$-module $\C_\alpha$ among the composition factors of $Q$ is 
\[
c_\alpha (Q)=\sum_{\beta\cleq\alpha}(-1)^{\ell(\alpha)-\ell(\beta)}{\rm Hilb}(Q_\beta,t).
\]
\end{lemma}

\begin{proof}
Let $0=Q_0\subset Q_1\subset\cdots\subset Q_k=Q$ be a composition series. We induct on the composition length $k$. The case $k=0$ is trivial. Assume $k>0$ below.

Suppose $Q/Q'\cong \C_\gamma$ for some $\gamma\models r+1$, i.e. there exists an element $z$ in $Q\setminus Q'$ satisfying
\[
T_iz\in\left\{\begin{array}{ll}
-z+Q', &{\rm if}\ i\in D(\gamma),\\
Q', &{\rm if}\ i\notin D(\gamma).
\end{array}\right. 
\]
Let $u$ be the longest element of the parabolic subgroup $W_{D(\gamma)^c}$ of $W$, and let 
\[
z'=T'_uz = \sum_{w\in W_{D(\gamma)^c}}T_wz.
\]
Since $D(w)\subseteq D(\gamma)^c$ for all $w$ in the above sum, we have $z'\in z+Q'$. Then any element in 
\[
Q=Q'\oplus\FF z=Q'\oplus\FF z'
\]
can be written as $y+az'$ for some $y\in Q'$ and $a\in\FF$. Since $D(u^{-1})=D(\gamma)^c$, one has $T_iz'=T_iT'_uz=0$ for all $i\notin D(\gamma)$.  Consider an arbitrary composition $\beta$ of $r+1$. 

If $\gamma\,\cleq\,\beta$ then for any  $i\notin D(\beta)$ we must have $i\notin D(\gamma)$ and thus $T_i(y+az')=T_iy$. It follows that $T_i(y+az')=0$ if and only if $T_iy=0$, i.e. $Q_\beta =Q'_\beta \oplus\FF z'$.

If $\gamma\!\not\!\!\cleq\,\beta$ then there exists $i\in D(\gamma)\setminus D(\beta)$. Using $z'\in z+Q'$ we have 
\[
T_i(y+az')=T_iy+aT_iz'\in -az+Q'. 
\]
If $T_i(y+az')=0$ then $a=0$. This implies $Q_\beta=Q'_\beta$.

It follows that
\begin{eqnarray*}
c_\alpha (Q)&=&\sum_{\beta\cleq\alpha}(-1)^{\ell(\alpha)-\ell(\beta)}{\rm Hilb}(Q_\beta,t)\\
&=&\sum_{\beta\cleq \alpha}(-1)^{\ell(\alpha)-\ell(\beta)}{\rm Hilb}(Q'_\beta,t)
+\sum_{\gamma\cleq\beta\cleq\alpha}(-1)^{\ell(\alpha)-\ell(\beta)}t^{\deg z'}\\
&=&c_\alpha (Q')+\delta_{\alpha\gamma}\cdot t^{\deg z'}.
\end{eqnarray*}

On the other hand, by induction hypothesis, the graded multiplicity of $\C_\alpha$ in the composition factors of $Q$ is also $c_\alpha (Q')+\delta_{\alpha\gamma}\cdot t^{\deg z'}$. Hence we are done.
\end{proof}

\begin{corollary}\label{cor:Ft}
If $Q$ is a finite dimensional filtered $H_n(0)$-module then
\[
\mathrm{Ch}_t(Q) = \sum_{\alpha\models r+1} {\rm Hilb}(Q_\alpha ,t) M_\alpha.
\]
Consequently, $\mathrm{Ch}_t(Q)\in\QSym[t]$ is a symmetric function, i.e. it lies in $\Sym[t]$, if and only if 
\[
{\rm Hilb}(Q_\alpha ,t)={\rm Hilb}(Q_\beta,t)\quad \textrm{whenever $\beta$ is a rearrangement of $\alpha$.}
\]
\end{corollary}

\begin{proof}
Apply inclusion-exclusion to the previous lemma.
\end{proof}

\begin{remark}
Lemma~\ref{Ft} and Corollary~\ref{cor:Ft} hold for an arbitrary field $\FF$.
\end{remark}

\begin{theorem}
The multiplicity of $\C_\alpha$ among the simple composition factors of $1^G_B$ is
\[ c_\alpha(1_B^G) = \sum_{w\in W: D(w^{-1})=D(\alpha)}q^{\ell(w)}. \]
\end{theorem}

\begin{proof}
Let $\overline B g$ be an element in $1^G_B$ where $g\in\FF G$. If it is annihilated by $T_j$ for all $j\in D(\alpha)^c$, then
$\overline{BwBg}=T_w(\overline Bg)=0$ for all $w$ with $D(w)\cap D(\alpha)^c\ne\emptyset$, and in particular, for all $w$ in $W_{D(\alpha)^c}\setminus\{1\}$. Hence 
\[
\overline Bg=\overline{BW_{D(\alpha)^c}B}g=\overline P_\alpha g\in 1^G_{P_\alpha}.
\]
Conversely, $T_j(\overline P_\alpha g)=T_jT'_{w_0(D(\alpha)^c)}(\overline Bg)=0$ for all $j\in D(\alpha)^c$. Therefore $(1^G_B)_\alpha=1^G_{P_\alpha}$. Applying Lemma \ref{Ft} to $1_B^G$ (trivially filtered) gives
\begin{eqnarray*}
c_\alpha(1_B^G) & = & \sum_{\beta\cleq\alpha} (-1)^{\ell(\alpha)-\ell(\beta)} |P_\alpha\backslash G| \\
& = & \sum_{\beta\cleq\alpha} (-1)^{\ell(\alpha)-\ell(\beta)} \sum_{w\in W:D(w^{-1})\subseteq D(\alpha)}|U_w| \\
& = &  \sum_{w\in W:D(w^{-1}) = D(\alpha)} q^{\ell(w)}.
\end{eqnarray*}
\end{proof}

\begin{corollary}
If $G=GL(n,\mathbb F_q)$ then $\mathrm{Ch}\left(1^G_B\right)=\widetilde H_{1^n}(x;q)$.
\end{corollary}

\begin{proof}
For $G=GL(n,\FF_q)$ we have $\ell(w)=\mathrm{inv}(w)$ and thus equation (\ref{eq:q-ribbon}) shows $c_\alpha(1_B^G)=r_\alpha(q)$. The result then follows from Corollary~\ref{cor:coinv}.
\end{proof}

\subsection{Decomposing the $G$-module $1_B^G$ by $0$-Hecke algebra action}

We consider the \emph{homology representations} $\chi_q^\alpha$ of $G$, which are the top homology of the \emph{type-selected Tits buildings} of $G$, for all compositions $\alpha\models r+1$. To give the explicit definitions, assume in this subsection that $1_{P_\alpha}^G$ is the $\FF$-span of the \emph{left} $P_\alpha$-cosets in $G$. Then $1_B^G$ admits a right $H_W(0)$-action defined by $g\overline B\cdot T_w = g\overline{BwB}$ for all $g\in G$ and $w\in W$. The left cosets $gP_\alpha$ for all $\alpha\models r+1$ form a poset under reverse inclusion, giving an (abstract) simplicial complex called the {\it Tits building} and denoted by $\Delta=\Delta(G,B)$. The type of a face $gP_\alpha$ is $\tau(gP_\alpha)=D(\alpha)$, and every chamber $gB$ has exactly one vertex of each type, i.e. $\Delta(G,B)$ is {\it balanced}. The chain complex of the {\it type-selected subcomplex}  
\[
\Delta_\alpha=\{F\in\Delta(G,B):\tau(F)\subseteq D(\alpha)^c\}.
\]
gives rise to an exact sequence
\begin{equation}\label{Chain}
0\rightarrow\chi^\alpha_q\rightarrow1^G_{P_\alpha} \xrightarrow{\partial}\bigoplus_{\beta\cleq_1\alpha} 1^G_{P_\beta}
\xrightarrow{\partial}\cdots\xrightarrow{\partial}1^G_G\rightarrow0
\end{equation} 
where $\beta\cleq_1\alpha$ means $\beta\models r+1$ and $D(\beta)=D(\alpha)\setminus\{i\}$ for some $i\in D(\alpha)$. The boundary maps are given by
\[
\partial: gP_\gamma \mapsto\sum_{\beta\cleq_1\gamma} \pm gP_\beta
\]
for all $\gamma\models r+1$. 
The following decomposition of (left) $G$-modules is well-known (see e.g. Smith~\cite{Smith}):
\begin{equation}\label{HomologyDecomposition}
1^G_B=\bigoplus_{\alpha\models r+1} \chi_q^\alpha.
\end{equation}
On the other hand, Norton's decomposition of the $0$-Hecke algebra $H_W(0)$ implies a decomposition of $1$ into primitive orthogonal idempotents, i.e.
\[
1=\sum_{\alpha\models r+1} h_\alpha  T_{w_0(\alpha)}T'_{w_0(\alpha^c)},\quad h_\alpha \in H_W(0).
\]
This decomposition of $1$ into primitive orthogonal idempotents is explicitly given by Berg, Bergeron, Bhargava and Saliola~\cite{Idempotent}, and is different from the one provided by Denton~\cite{Denton}. By the right action of $H_W(0)$ on $1^G_B$, we have another decomposition of $G$-modules:
\begin{equation}\label{BuildingDecomposition} 
1^G_B=\bigoplus_{\alpha\models r+1} 1^G_B\,h_\alpha T_{w_0(\alpha)}T'_{w_0(\alpha^c)}.
\end{equation} 

\begin{proposition}
The two $G$-module decompositions (\ref{HomologyDecomposition}) and (\ref{BuildingDecomposition}) are the same. 
\end{proposition} 

\begin{proof}
Comparing (\ref{HomologyDecomposition}) with (\ref{BuildingDecomposition}) one sees that it suffices to show $1^G_Bh_\alpha T_{w_0(\alpha)}T'_{w_0(\alpha^c)}\subseteq\chi^\alpha_q.$
Assume 
\[
\overline Bh_\alpha T_{w_0(\alpha)}=\sum_i g_i\overline B,\quad g_i\in G.
\]
For any $\beta\models r+1$ we have 
\[
\overline BT'_{w_0(\beta^c)}=\overline{BW_{D(\beta)^c}B}=\overline P_\beta.
\]
Hence 
\[
\overline Bh_\alpha T_{w_0(\alpha)}T'_{w_0(\alpha^c)} = \sum_i g_i\overline P_\alpha \in 1_{P_\alpha}^G
\]
and
\[
\partial \left(\overline Bh_\alpha T_{w_0(\alpha)}T'_{w_0(\alpha^c)} \right)
= \sum_{\beta\cleq_1\alpha} \pm\sum_i g_i\overline P_\beta
= \sum_{\beta\cleq_1\alpha} \pm\overline Bh_\alpha T_{w_0(\alpha)}T'_{w_0(\beta^c)}.
\]
If $\beta\cleq_1\alpha$ then there exists $i\in D(\alpha)\cap D(\beta)^c$, and thus $T_{w_0(\alpha)}T'_{w_0(\beta^c)}=0$. Therefore we are done.
\end{proof}

One sees from (\ref{BuildingDecomposition}) that $\chi_q^\alpha$ is in general not a right $H_W(0)$-submodule of $1_B^G$. However, when $G=GL(n,\FF_q)$, one has that $\chi_q^{(n)}$ is the trivial representation and $\chi_q^{(1^n)}$ is the Steinberg representation of $G$ \cite{Khammash}, and both are right (isotypic) $H_W(0)$-modules.




\section{Coinvariant algebra of $(G,B)$}\label{sec:CoinvGB}


In this section we again assume $\FF$ is a field of characteristic $p>0$ and study the action of the $0$-Hecke algebra $H_n(0)$ on the coinvariant algebra $\FF[\bx]^B/(\FF[\bx]^G_+)$ of the pair $(G,B)$, where $G=GL(n,\FF_q)$ with $q$ being a power of $p$ and $B$ is the Borel subgroup of $G$.

Given a right $\FF G$-module $M$, there is an isomorphism
\begin{eqnarray*}
{\rm Hom}_{\FF G}(1_B^G,M)&\xrightarrow{\sim}& M^B,\\
\phi&\mapsto& \phi(\overline B)
\end{eqnarray*}
with inverse map given by $\phi_m(\overline B)=m$ for all $m\in M^B$. The left $H_n(0)$-action $T_w\overline B=\overline{BwB}$ on $1_B^G$ commutes with the right $G$-action and induces a left action on ${\rm Hom}_{\FF G}(1_B^G,M)$ by
\[
T_w(\phi)(\overline B)=\phi(T_{w^{-1}}\overline B) = \phi_m(\overline{Bw^{-1}B}).
\]
Hence we have a left $H_n(0)$-action on $M^B$ by
\[
T_w(m)=T_w(\phi_m)(\overline B)=\phi_m(\overline{Bw^{-1}B})
=\phi_m(\overline B w^{-1}\overline U_{w^{-1}})=mw^{-1}\overline U_{w^{-1}}.
\]

The group $G$ has a left action on $\FF[\bx]$ by linear substitution, and this can be turned into a right action by $f\cdot g=g^{-1}f$ for all $f\in\FF[\bx]$ and $g\in G$. Thus $H_n(0)$ has a left action on $\FF[\bx]^B$ by 
\[
T_w(f)=f\cdot w^{-1}\overline U_{w^{-1}}=\overline U_w wf,\quad \forall f\in \FF[\bx]^B.
\]
This action preserves the grading, and leaves the ideal $(\FF[\bx]^G_+)$ invariant: if $h_i\in\FF[\bx]^G_+$, $f_i\in\FF[\bx]^B$, then
\[
T_w\left(\sum_i h_if_i\right)=\overline U_w w \left(\sum_i h_if_i\right)
=\sum_i h_i\overline U_w w (f_i).
\]
Hence the coinvariant algebra $\FF[\bx]^B/(\FF[\bx]^G_+)$ of $(G,B)$ becomes a graded $H_n(0)$-module. 

\begin{lemma}\label{QI}
If $Q=\FF[\bx]^B/(\FF[\bx]^G_+)$ and $\alpha$ is a composition of $n$,
then 
\[
Q_\alpha := \bigcap_{j\in D(\alpha)^c}\ker T_j = \FF[\bx]^{P_\alpha}/(\FF[\bx]^G).
\]
\end{lemma}

\begin{proof}
If $f\in\FF[\bx]^{P_\alpha}$, then for all $j\notin D(\alpha)$ one has $U_{s_j}s_j\subseteq P_\alpha$ and hence
\[ T_jf=\overline U_{s_j}s_jf=|U_{s_j}|\cdot f=qf=0. \]
Conversely, a $B$-invariant polynomial $f$ gives rise to a $P_\alpha $-invariant polynomial 
\[
\sum_{gB\in P_\alpha /B}gf=\sum_{w\in W_{D(\alpha)^c}}\overline U_wwf=T'_{w_0(D(\alpha)^c)}f.
\]
If $T_jf$ belongs to the ideal $(\FF[\bx]^G_+)$ for all $j\notin D(\alpha)$, so does  $T_wf\in (\FF[\bx]^G_+)$ for all $w\in W_{D(\alpha)^c}\setminus\{1\}$. Thus $T'_{w_0(\alpha^c)} f -f \in (\FF[\bx]^G_+)$ and we are done. 
\end{proof}

\begin{theorem}
The graded quasisymmetric characteristic of the $H_n(0)$-module $\FF[\bx]^B/(\FF[\bx]^G_+)$ is
\[ 
\mathrm{Ch}_t\left(\FF[\bx]^B/(\FF[\bx]^G_+)\right) = \sum_{\alpha\models n} \qbin{n}{\alpha}{q,t} M_\alpha = \sum_{\alpha\models n} r_\alpha (q,t)F_\alpha.
\]
\end{theorem}

\begin{proof}
Let $Q=\FF[\bx]^B/(\FF[\bx]^G_+)$ and let $\alpha\models n$. It follows from Lemma~\ref{QI} that
\[
\mathrm{Hilb}(Q_\alpha,t)={\rm Hilb}\left(\FF[\bx]^{P_\alpha }/(\FF[\bx]^G_+),t\right) = \qbin{n}{\alpha}{q,t}.
\]
Thus
\[
c_\alpha (Q)=\sum_{\beta\cleq\alpha} (-1)^{\ell(\alpha)-\ell(\beta)} \qbin{n}{\beta}{q,t}=r_\alpha (q,t).
\]
Then the result follows immediately from Lemma~\ref{Ft}.
\end{proof}


\section{Cohomology ring of Springer fibers}\label{sec:Springer}

In Section~\ref{SCoinvariantsA} we showed that the coinvariant algebra of $\SS_n$ is an $H_n(0)$-module whose graded quasisymmetric characteristic is the modified Hall-Littlewood symmetric function indexed by the partition $1^n$. Now we generalize this result to partitions of hook shapes. 

Throughout this section a partition of $n$ is denoted by $\mu=(\mu_1,\ldots,\mu_n)$, where $0\leq\mu_1\leq\cdots\leq\mu_n$. Denote $n(\mu)=\mu_{n-1}+2\mu_{n-2}+\cdots+(n-1)\mu_1$. One can view $\mu$ as a composition by dropping all the zero parts of $\mu$. Then $n(\mu)=\maj(\mu)$, where $\maj(\alpha)=\sum_{i\in D(\alpha)}i$, for all compositions $\alpha$.

Let $V$ be an $n$-dimensional complex vector space.  Fix a nilpotent matrix $X_\mu$ whose Jordan blocks are of size $\mu_1,\ldots,\mu_\ell$. The \emph{Springer fiber} $\mathcal F_\mu$ corresponding to the partition $\mu$ is the variety of all flags $0\subset V_1\subset\cdots\subset V_n=V$ of subspaces $V_i\subseteq V$ satisfying $\dim V_i=i$ and $X_\mu(V_i)\subseteq V_{i-1}$. The cohomology ring of $\mathcal F_\mu$ is isomorphic to the ring $\CC[\bx]/I_\mu$ for certain homogeneous ideal $I_\mu$, and carries an $\SS_n$-action that can be obtained from the $\SS_n$-action on $\CC[\bx]$. In particular, if $\mu=1^n$ then $\mathcal F_\mu$ is the flag variety $G/B$ and $\CC[\bx]/I_\mu$ is the coinvariant algebra of $\SS_n$. 

\begin{theorem}[Hotta-Springer~\cite{HottaSpringer}, Garsia-Procesi~\cite{GarsiaProcesi}]\label{thm:Springer}
The graded Frobenius characteristic of $\CC[\bx]/I_\mu$ is the modified Hall-Littlewood symmetric function 
\[
\widetilde H_\mu(x;t)= \sum_\lambda t^{n(\mu)} K_{\lambda\mu}(t^{-1})s_\lambda
\]
where $K_{\lambda\mu}(t)$ is the Kostka-Foulkes polynomial.
\end{theorem}

To find an analogous result for the $0$-Hecke algebras, let $R_\mu:=\FF[\bx]/I_\mu$ where $\FF$ is an arbitrary field, and we consider the question of when the $H_n(0)$-action on $\FF[\bx]$ preserves the ideal $I_\mu$. Recall the following construction of $I_\mu$ by Tanisaki~\cite{Tanisaki}. Let the conjugate of a partition $\mu$ of $n$ be $\mu'=(0\leq\mu'_1\leq\cdots\leq\mu'_n)$. Note that the \emph{height} of the Ferrers diagram of $\mu$ is $h=h(\mu):=\mu'_n$. Let
\[
d_k(\mu)=\mu'_1+\cdots+\mu'_k,\quad k=1,\ldots,n.
\]
Then the ideal $I_\mu$ is generated by
\begin{equation}\label{I_mu}
\{e_r(S): k\geq r> k-d_k(\mu),\ |S|=k,\ S\subseteq\{x_1,\ldots,x_n\}\}
\end{equation}
where $e_r(S)$ is the $r$-th elementary symmetric function in the set $S$ of variables. See also Garsia and Procesi~\cite{GarsiaProcesi}.

\begin{proposition}
The Demazure operators preserve the ideal $I_\mu$ if and only if $\mu$ is a hook.
\end{proposition}

\begin{proof}
First assume $\mu=(0^{n-h},1^{h-1},n-h+1)$ is a hook. Then $\mu'=(0^{h-1},1^{n-h},h)$ and so
\[
(1-d_1(\mu),2-d_2(\mu),\ldots,n-d_n(\mu))=(1,2,\ldots,h-1,h-1,\ldots,h-1,0).
\]
It follows that the ideal $I_\mu$ is generated by the elementary symmetric functions $e_1,\ldots,e_n$, together with the following partial elementary symmetric functions
\[
\{e_r(S):r=h,\ldots,k,\ S\subseteq\{x_1,\ldots,x_n\},\ |S|=k,\ k=h,\ldots,n-1\}.
\]
This implies that $I_\mu$ can be generated by $e_1,\ldots,e_n$ and all monomials in
\[
\mathcal M_h=\{x_{i_1}\cdots x_{i_h}:0\leq i_1<\cdots<i_h\leq n\}.
\]
By (\ref{pi}), if $x_{i_1}\cdots x_{i_h}$ is in $\mathcal M_h$ and $f$ is an arbitrary monomial, then $\overline\pi_i(x_{i_1}\cdots x_{i_h}f)$ is divisible by some element in $\mathcal M_h$. Thus the ideal $I_\mu$ is preserved by the Demazure operators.

Now assume $\mu$ is not a hook. Then $\mu'_{n-1}\geq 2$ and thus
\begin{eqnarray*}
k-d_k(\mu) &=& k-n+n-d_k(\mu)\\
& =& k-n+\mu'_n+\mu'_{n-1}+\cdots+\mu'_{k+1} \\
& \geq & k-n+\mu'_n+2+1+\cdots+1 \\
& = & \mu'_n=h
\end{eqnarray*}
for $k=n-2,\ldots,n-\mu_1+1$.
One also sees that
\[
k-d_k(\mu)=\left\{
\begin{array}{cc}\
0, & k=n, \\
h-1, & k=n-1,\\
k, & n-\mu_1\geq k\geq 1.
\end{array}\right.
\]
Thus the only elements in the generating set (\ref{I_mu}) that have degree no more than $h$ are $e_1,\ldots,e_h$ and those $e_h(S)$ with $|S|=n-1$.

Suppose to the contrary that $I_\mu$ is preserved by Demazure operators. Since 
\[
e_h(x_1,\ldots,x_{n-1})=x_{n-1}e_{h-1}(x_1,\ldots,x_{n-2})+e_h(x_1,\ldots,x_{n-2})\in I_\mu
\]
we have
\[
\overline \pi_{n-1} e_h(x_1,\ldots,x_{n-1})=x_n e_{h-1}(x_1,\ldots,x_{n-2})\in I_\mu
\]
and thus
\[
e_h(x_1,\ldots,x_{n-2},x_n)-\overline \pi_{n-1} e_h(x_1,\ldots,x_{n-1})=e_h(x_1,\ldots,x_{n-2})\in I_\mu.
\]
Repeating this process one obtains $e_h(x_1,\ldots,x_h)=x_1\cdots x_h \in I_\mu$. Then applying the Demazure operators to $x_1\cdots x_h$ gives $x_{i_1}\cdots x_{i_h}\in I_\mu$ whenever $1\leq i_1<\cdots<i_h\leq n$. Considering the degree we have 
\[
x_{i_1}\cdots x_{i_h}=\sum_{i=1}^h f_i e_i + \sum_{|S|=n-1} c_S e_h(S)
\]
where $f_i\in\FF[\bx]$ is homogeneous of degree $h-i$ and $c_S\in\FF$. It follows that $U\subseteq U'$, where $U$ and $U'$ are the $\FF$-subspaces of the coinvariant algebra $\FF[\bx]/(e_1,\ldots,e_n)$ spanned by 
\[
\{x_{i_1}\cdots x_{i_h}:1\leq i_1<\cdots<i_h\leq n-1\} \quad {\rm and} \quad \{e_h(S):|S|=n-1\}.
\]
It is well-known that all divisors of $x_1^{n-1}x_2^{n-2}\cdots x_{n-1}$ form a basis for $\FF[\bx]/(e_1,\ldots,e_n)$, i.e. the \emph{Artin basis}. Thus
\[
n = {n\choose n-1} \geq \dim U' \geq \dim U = {n-1\choose h}\geq{n-1\choose 2} \]
where the last inequality follows from $2\leq h\leq n-2$ since $\mu$ is a hook. Therefore we must have $n=4$ and $h=2$. In that case it is easy to check $x_1x_2\notin I_{(2,2)}$. Hence we are done.
\end{proof}

\begin{theorem}\label{thm:Rmu}
Assume $\mu=(0^{n-h},1^{h-1},n-h+1)$ is a hook and view it as a composition by removing all zeros. Then the $H_n(0)$-module $R_\mu$ is a direct sum of the projective indecomposable $H_n(0)$-modules $\P_\alpha$ for all compositions $\alpha\cleq\mu$, i.e.
\[
R_\mu \cong \bigoplus_{\alpha\cleq \mu} \P_\alpha.
\]
\end{theorem}

\begin{proof}
By the proof of the previous proposition, $I_\mu$ is generated by $e_1,\ldots,e_n$ and 
\[
\mathcal M_h=\{x_{i_1}\cdots x_{i_h}:1\leq i_1<\cdots<i_h\leq n\}.
\]
Thus $R_\mu$ is the quotient of $\FF[\bx]/(\FF[\bx]^{\SS_n}_+)$ by its ideal generated by $\mathcal M_h$. By Theorem~\ref{CoinvariantsA}, it suffices to show that $\FF[\bx]/I_\mu$ has a basis given by
\begin{equation}\label{eq:mu-basis}
\{\overline \pi_w x_{D(w)}: w\in\mathfrak S_n, D(w)\subseteq D(\mu)\}.
\end{equation}

If $D(w)\not\subseteq D(\mu)=\{1,2,\ldots,h-1\}$, i.e. $w$ has a descent $i\geq h$, then $x_{D(w)}$ contains at least $h$ distinct variables, and so do all monomials in $\overline\pi_w x_{D(w)}$ by (\ref{pi}). Thus $\FF[\bx]/I_\mu$ is spanned by (\ref{eq:mu-basis}).

To show (\ref{eq:mu-basis}) is linearly independent, assume
\[
\sum_{D(w)\subseteq D(\mu)} c_w\overline\pi_w x_{D(w)}\in I_\mu,\quad c_w\in\FF.
\]
For any polynomial $f\in\FF[\bx]$, let $[f]_h$ be the polynomial obtained from $f$ by removing all terms divisible by some element in $\mathcal M_h$. It follows that
\[
\sum_{D(w)\subseteq D(\mu)} c_w[\overline\pi_w x_{D(w)}]_h
\in (e_1,\ldots,e_{h-1}).
\]
If $D(w)\subseteq D(\mu)=[h-1]$ then the leading term of $[\overline\pi_w x_{D(w)}]_h$ under ``$\prec$'' is still the descent monomial $wx_{D(w)}$. By Lemma~\ref{lem:basis},
\[
\left\{[\overline\pi_w x_{D(w)}]_h: w\in\mathfrak S_n, D(w)\subseteq D(\mu)\right\}
\]
gives a linearly independent set in $\FF[\bx]/(\FF[\bx]^{\SS_n}_+)$. This forces $c_w=0$ whenever $D(w)\subseteq D(\mu)$.
\end{proof}

By work of Bergeron and Zabrocki~\cite{BZ}, the modified Hall-Littlewood functions $\widetilde H_\mu(x;t)$ have the following noncommutative analogue lying in $\NSym[t]$ for all compositions $\alpha$:
\[
\widetilde{\mathbf{H}}_\alpha(\mathbf x;t):=\sum_{\beta\cleq\alpha} t^{\maj(\beta)} \bs_\beta.
\] 

\begin{corollary}
Assume $\mu$ is a hook. Then 
\[
\mathbf{ch}_t(R_\mu)=\sum_{\alpha\cleq\mu} t^{{\rm maj}(\alpha)}\bs_\alpha=\widetilde{\mathbf{H}}_\mu(\bx;t),
\]
\[
\mathrm{Ch}_t(R_\mu)=\sum_{\alpha\cleq\mu}
t^{{\rm maj}(\alpha)} s_\alpha = \widetilde H_\mu(x;t).
\]
\end{corollary}

\begin{proof}
Theorem~\ref{thm:Rmu} immediately implies the graded noncommutative characteristic of $R_\mu$, whose commutative image is the graded quasisymmetric characteristic of $R_\mu$. One can check that $\mathrm{Ch}_t(R_\mu) = \widetilde H_\mu(x;t)$ when $\mu$ is a hook, by using Haglund's combinatorial formula~\cite{HHL} for the modified Macdonald polynomials $\widetilde H_\mu(x;q,t)$ and taking $q\to0$.
\end{proof}

\begin{remark}\label{rem:Rmu}
(i) We recover the results in Section~\ref{SCoinvariantsA} by taking $\mu=1^n$.

\noindent(ii) Using certain difference operators, Hivert~\cite{Hivert} defined a noncommutative analogue of the Hall-Littlewood functions, which is in general different from the noncommutative analogue of Bergeron and Zabrocki. However, they are the same when $\mu$ is a hook! 

\noindent (iii) It is not clear to the author why the results are nice only in the hook case, except for a naive explanation: the hooks are the only diagrams that belong to both the family of the Ferrers diagrams of partitions and the family of ribbon diagrams of compositions.
\end{remark}

\section{Questions for future research}\label{SQ}

\subsection{Equidistribution of the inversion number and major index}
The equidistribution of inv and maj was first proved on permutations of multisets by P.A. MacMahon in the 1910s; applying an inclusion-exclusion would give their equidistribution on inverse descent classes of $\mathfrak S_n$. However, the first proof for the latter result appearing in the literature was by Foata and Sch\"utzenberger \cite{FS} in 1970, using a bijection constructed earlier by Foata \cite{Foata}. Is there an algebraic proof from the $(q,t)$-bigraded characteristic of $\FF[{\bx}]/(\FF[{\bx}]^{\mathfrak S_n}_+)$, which is given in Corollary~\ref{cor:coinv} (i) and involves inv, maj, and inverse descents?

\subsection{Decompositions of $1_B^G$ and $\FF[{\bx}]^B/(\FF[{\bx}]^G_+)$}
In \S~\ref{SBG} we studied an $H_W(0)$-action on the flag variety $1_B^G$ and found its simple composition factors, but we do \emph{not} know the decomposition of $1_B^G$ into indecomposable $H_W(0)$-modules.  Assume $G=GL(n,\FF_q)$ below. Computations show that $1_B^G$ is in general \emph{not} projective, 
although its quasisymmetric characteristic is always symmetric.




The coinvariant algebra $\FF[{\bx}]^B/(\FF[{\bx}]^G_+)$ is \emph{not} a projective $H_n(0)$-module either, since its graded quasisymmetric characteristic is \emph{not} even symmetric (see the definition of the $(q,t)$-multinomial coefficients). To find its decomposition, it will be helpful to know more (nonprojective) indecomposable $H_n(0)$-modules (there are infinitely many, and some were studied by Duchamp, Hivert, and Thibon~\cite{NCSFVI}).
 
Another question is to find a $q$-analogue of the Demazure operators, which might give another $H_n(0)$-action on $\FF[{\bx}]^B/(\FF[{\bx}]^G_+)$.

\subsection{Coincidence of Frobenius type characteristics}
For $G=GL(n,\CC)$, the complex flag variety $1_B^G$ has its cohomology ring isomorphic to the coinvariant algebra of $\SS_n$, whose graded Frobenius characteristic and graded quasisymmetric characteristic both equal the modified Hall-Littlewood symmetric function $\widetilde H_{1^n}(x;t)$. For $G=GL(n,\FF_q)$, the flag variety $1_B^G$ itself, when defined over a field of characteristic $p\mid q$, is also an $H_n(0)$-module whose quasisymmetric characteristic equals $\widetilde H_{1^n}(x;q)$. Is there a better explanation for the coincidence of these Frobenius type characteristics?


\end{document}